\documentclass[11pt]{amsart}

\usepackage{subfiles}
\usepackage{euscript}
\usepackage{tabu}
\usepackage[margin=1in]{geometry} 		             		
\usepackage{graphicx}
\usepackage{faktor}
\usepackage{xcolor}
\usepackage{amsmath,amsthm,amssymb}
\usepackage{mathrsfs}
\usepackage{stmaryrd}
\usepackage{tikz}
\usepackage{tikz-cd}
\usepackage{accents}
\usepackage{upgreek}
\usepackage{enumitem}
\usepackage{bm}
\usepackage{mathtools}
\usepackage[all]{xy}
\usepackage{caption}

\usepackage{url}
\usepackage{float}
\usepackage{todonotes}
\usepackage{bbm}
\usepackage{colonequals}
\usepackage{longtable}

\usepackage[full]{textcomp}
\usepackage[cal=cm]{mathalfa}
\usepackage{xparse}
\usepackage{comment}
\usepackage{cite}

\tikzset{
  commutative diagrams/.cd, 
  arrow style=tikz, 
  diagrams={>=stealth}
}

\theoremstyle{definition}
\newtheorem{innercustomthm}{Theorem}
\newenvironment{customthm}[1]
  {\renewcommand\theinnercustomthm{#1}\innercustomthm}
  {\endinnercustomthm}
  
\makeatletter
\def\@tocline#1#2#3#4#5#6#7{\relax
  \ifnum #1>\c@tocdepth 
  \else
    \par \addpenalty\@secpenalty\addvspace{#2}%
    \begingroup \hyphenpenalty\@M
    \@ifempty{#4}{%
      \@tempdima\csname r@tocindent\number#1\endcsname\relax
    }{%
      \@tempdima#4\relax
    }%
    \parindent\z@ \leftskip#3\relax \advance\leftskip\@tempdima\relax
    \rightskip\@pnumwidth plus4em \parfillskip-\@pnumwidth
    #5\leavevmode\hskip-\@tempdima
      \ifcase #1
       \or\or \hskip 1em \or \hskip 2em \else \hskip 3em \fi%
      #6\nobreak\relax
    \dotfill\hbox to\@pnumwidth{\@tocpagenum{#7}}\par
    \nobreak
    \endgroup
  \fi}
\makeatother

\makeatletter
\DeclareRobustCommand{\cev}[1]{%
  \mathpalette\do@cev{#1}%
}
\newcommand{\do@cev}[2]{%
  \fix@cev{#1}{+}%
  \reflectbox{$\m@th#1\vec{\reflectbox{$\fix@cev{#1}{-}\m@th#1#2\fix@cev{#1}{+}$}}$}%
  \fix@cev{#1}{-}%
}
\newcommand{\fix@cev}[2]{%
  \ifx#1\displaystyle
    \mkern#23mu
  \else
    \ifx#1\textstyle
      \mkern#23mu
    \else
      \ifx#1\scriptstyle
        \mkern#22mu
      \else
        \mkern#22mu
      \fi
    \fi
  \fi
}

\makeatother

\usetikzlibrary{calc}
\usetikzlibrary{fadings}
\usetikzlibrary{decorations.pathmorphing}
\usetikzlibrary{decorations.pathreplacing}
\usetikzlibrary{shapes}

\newcounter{marginnote}
\DeclareMathAlphabet{\mathpzc}{OT1}{pzc}{m}{it}

\definecolor{sebdarkgreen}{rgb}{0.019,0.317,0.149}
\definecolor{sebgreen}{RGB}{36,157,55}
\definecolor{seblightgreen}{rgb}{0.784,0.952,0.780}
\definecolor{sebblue}{RGB}{0,123,194}
\definecolor{seblightblue}{RGB}{194,242,255}

\usepackage[pagebackref]{hyperref}
\hypersetup{
  colorlinks   = true,          
  urlcolor     = blue,          
  linkcolor    = purple,          
  citecolor   = sebblue             
}

\theoremstyle{definition}
\newtheorem{theorem}{Theorem}[section]
\newtheorem*{theorem*}{Theorem}

\newtheorem{conjecture}[theorem]{Conjecture}
\newtheorem*{conjecture*}{Conjecture}
\newtheorem{corollary}[theorem]{Corollary}
\newtheorem{lemma}[theorem]{Lemma}
\newtheorem{proposition}[theorem]{Proposition}
\newtheorem{remark}[theorem]{Remark}

\newtheorem*{runningexample*}{Running example}

\newtheorem*{aside*}{Aside}

\newtheorem{definition}[theorem]{Definition}
\newtheorem{example}[theorem]{Example}

\newtheorem{proposition-definition}[theorem]{Proposition-Definition}

  \theoremstyle{theorem}
\newtheorem{innercustomconj}{Conjecture}
\newenvironment{customconj}[1]
  {\renewcommand\theinnercustomconj{#1}\innercustomconj}
  {\endinnercustomconj}

  \theoremstyle{theorem}
\newtheorem{innercustomcor}{Corollary}
\newenvironment{customcor}[1]
  {\renewcommand\theinnercustomcor{#1}\innercustomcor}
  {\endinnercustomcor}
  
\DeclareMathOperator{\Pic}{Pic}

\newcommand{\NN}{\mathbf{N}}

\newcommand{\Gm}{\mathbb{G}_{\rm{m}}}
\newcommand{\Glog}{\mathbb{G}_{\rm{log}}}
\newcommand{\Gtrop}{\mathbb{G}_{\rm{trop}}}

\newcommand{\lbb}{[\![}
\newcommand{\rbb}{]\!]}

\newcommand{\ol}[1]{\overline{#1}}

\newcommand{\bcd}{\begin{center}\begin{tikzcd}}
\newcommand{\ecd}{\end{tikzcd}\end{center}}

\newcommand{\Aaff}{\mathbb{A}}

\newcommand{\PP}{\mathbb{P}}
\newcommand{\OO}{\mathcal{O}}
\renewcommand{\NN}{\mathbb{N}}

\newcommand{\ZZ}{\mathbb{Z}}

\newcommand{\kfield}{\bm{k}}
\newcommand{\dvr}{\Delta}

\newcommand{\Lcal}{\mathcal{L}}

\newcommand{\Mcal}{\mathcal{M}}
\newcommand{\Hcal}{\mathcal{H}}

\newcommand{\Bcal}{\mathcal{B}}

\newcommand{\Fcal}{\mathcal{F}}
\newcommand{\Ccal}{\mathcal{C}}
\newcommand{\oCcal}{\overline{\mathcal{C}}}

\newcommand{\Pcal}{\mathcal{P}}
\newcommand{\Tau}{T}
\newcommand{\oPcal}{\overline{\mathcal{P}}}
\newcommand{\oPsi}{\bar\Psi}
\newcommand{\Qcal}{\mathcal{Q}}

\newcommand{\blambda}{\bar{\lambda}}

\newcommand{\Mbar}{\ol{\Mcal}}
\newcommand{\oP}{\overline{P}}

\newcommand{\val}{\text{val}}

\newcommand{\Spec}{\operatorname{Spec}}

\NewDocumentCommand{\compatibilitydatum}{m m m m m m O{} O{} O{}}{
\begin{equation*} \begin{tikzcd}[ampersand replacement=\&]
  \: \arrow{r} \& {#1} \arrow{r} \arrow{d}{#7} \& {#2} \arrow{r} \arrow{d}{#8} \& {#3} \arrow{r}{[1]} \arrow{d}{#9} \& \: \\
  \: \arrow{r} \& {#4} \arrow{r} \& {#5} \arrow{r} \& {#6} \arrow{r} \& \:
\end{tikzcd} \end{equation*}}

\NewDocumentCommand{\commutingsquare}{m m m m o O{} O{} O{} O{}}{
\begin{equation}\begin{tikzcd}[ampersand replacement=\&] \label{#5}
  #1 \arrow{r}{#6} \arrow{d}{#7} \& #2 \arrow{d}{#8} \\
  #3 \arrow{r}{#9} \& #4
\end{tikzcd}\IfValueTF{#5}{\label{#5}}{} \end{equation}}

\NewDocumentCommand{\cartesiansquare}{m m m m O{} O{} O{} O{}}{
\begin{equation*}\begin{tikzcd}[ampersand replacement=\&]
  #1 \arrow{r}{#5} \arrow{d}{#6} \arrow[dr, phantom, "\square"] \& #2 \arrow{d}{#7} \\
  #3 \arrow{r}{#8} \& #4
\end{tikzcd} \end{equation*}}

\NewDocumentCommand{\cartesiansquarelabel}{m m m m m O{} O{} O{} O{}}{
\begin{tikzcd}[ampersand replacement=\&]
  #1 \arrow{r}{#6} \arrow{d}{#7} \arrow[dr, phantom, "\square"] \& #2 \arrow{d}{#8} \\
  #3 \arrow{r}{#9} \& #4
\end{tikzcd}\IfValueTF{#5}{\label{#5}}{}
}

\NewDocumentCommand{\triangleofspaces}{m m m O{} O{} O{}}{
\begin{tikzcd} [ampersand replacement=\&]
#1 \arrow{r}{#4} \arrow[bend right]{rr}{#5} \& #2 \arrow{r}{#6} \& #3
\end{tikzcd}}

\newcommand{\on}{\operatorname}

\newcommand{\oC}{\overline{C}}
\newcommand{\opsi}{\overline{\psi}}
\newcommand{\omd}{^{\omega}}
\newcommand{\homs}{\mathcal{H}om}

\usepackage{hyphenat}
\begin{document}
 
\title{Hyperelliptic Gorenstein curves and logarithmic differentials}

\author{Luca Battistella}
\address{Institut für Mathematik, Humboldt-Universität zu Berlin, Germany}
\email{luca.battistella@hu-berlin.de}

\author{Sebastian Bozlee}
\address{Tufts University, Medford, MA, United States of America}
\email{Sebastian.Bozlee@tufts.edu}

\begin{abstract}
We produce a flexible tool for contracting subcurves of logarithmic hyperelliptic curves, which is local around the subcurve and commutes with arbitrary base-change. As an application, we prove that hyperelliptic multiscale differentials determine a sequence of Gorenstein contractions of the underlying nodal curve, whose dualising bundle they descend to generate. This is the first piece of evidence for a more general conjecture about limits of differentials.
\end{abstract}

\maketitle
\setcounter{tocdepth}{1}
\tableofcontents

\section*{Introduction}
Moduli spaces of differentials on Riemann surfaces have undergone wide and deep investigation at the interface between dynamics, topology, and algebraic geometry \cite{EMM,Filip}.

Various questions in Teichm\"uller theory can be interpreted in intersection-theoretic terms on a compact moduli space \cite{Mirzakhani,CMSZ}. In order to compactify strata of differentials, the curve $C$ should be allowed to degenerate: in the limit a smooth curve can become nodal, and the differential $\eta$ can vanish on a subcurve $C_{<0}$. Rescaling the differential appropriately, though, it is possible to extract more information, namely a meromorphic differential $\eta_{<0}$ on the subcurve $C_{<0}$ (possibly vanishing on a subcurve $C_{<-1}$, and so on). The dual graph of the nodal curve appears thus to be aligned by the generic vanishing order of the differential. Moreover, the order of zeroes and poles of the various meromorphic differentials define a conewise-linear function $\lambda$ with integer slopes on the dual graph. This is, roughly speaking, a \emph{generalised multiscale differential} \cite{BCGGM}. See also \cite{Gendron,FarkasPandharipande} for different approaches to compactifying strata of differentials.

Moduli spaces of generalised multiscale differentials are typically not irreducible, and the locus of smoothable differentials has been characterized in terms of the so-called \emph{global residue condition}: a zero-sum condition on residues of the differential at poles belonging to different irreducible components of the curve, which are connected through components at higher levels \cite{BCGGM16}.

The compactification is intrinsically logarithmic \cite{Chen2,Tale}. The conewise linear function is indeed a section of the characteristic sheaf of a log structure on the curve. It is a tropical canonical differential in the sense that it belongs to the tropical canonical linear series. Even for these purely combinatorial data, the moduli space is not in general irreducible (nor pure-dimensional); the locus of smoothable (\emph{realisable}) tropical differentials has been described explicitly in \cite{MUW}.

With the logarithmic approach providing a purely algebraic point of view on multiscale differentials, identifying the main component is the only outstanding problem towards a characteristic-free understanding of moduli spaces of differentials. We state a conjecture, originally due to D. Ranganathan and J. Wise, to the effect that smoothable differentials should be exactly those descending to a sequence of Gorenstein contractions of the curve.

\begin{customconj}{G}[$\approx$ {Conjecture~\ref{conj}}] \label{conjG}\leavevmode
Let $(C,\eta)$ be a generalised multiscale/log rubber differential (up to scaling), 
and let $\blambda$ denote its tropicalization. Then $\eta$ is smoothable if and only if \begin{enumerate}[label=(\roman*)]
    \item every level truncation $\lambda_i$ of $\blambda$ (as in \S \ref{sec:levels}) is a realisable tropical differential;
    \item there exists a reduced Gorenstein contraction $\sigma\colon C\to\oC_i$ such that $\sigma^*\omega_{\oC_i}=\omega_C(\lambda_i)$;
    \item the differential $\eta_i$ at level $i$ descends to a local generator of $\omega_{\oC_i}$. 
\end{enumerate}
\end{customconj}

The conjecture is motivated by work on stable maps \cite{RSPW1,RSPW2,BCM,BNR1,Wanlong,BatCar,BCquartics}. The connection between Gorenstein singularities and the (algebraic and tropical) geometry of differentials was first  evidenced in \cite{Bat19}. It appears from these works (on curves of genus one and two) that Brill--Noether theory, intended as the study of special linear series on curves, plays a key role in the construction of alternative compactifications of the moduli space of curves, embedded or not. In this paper, we explore this connection in the more general framework of hyperelliptic curves, and study Conjecture \ref{conjG} in this special case. In forthcoming work we will present applications of our construction to the birational geometry of the moduli space of hyperelliptic curves 
\cite{SmythTowards,Smyth,Fedorchuk,Bat19,BKN, BarrosMullane,BlankersBozlee}.

Strata of differentials are known to have at most three connected components \cite{KontsevichZorich}. One of them parametrises \emph{hyperelliptic differentials}, i.e. differentials on hyperelliptic curves that are anti-invariant under the hyperelliptic involution. Even after compactifying, this component is already irreducible \cite[\S 5]{Chen2}, hence the above conjecture postulates that every hyperelliptic multiscale differential should come from a Gorenstein contraction. This is indeed what we prove; the bulk of the paper consists of the construction of such a contraction. The combinatorial data we need is a cutoff of the tropicalisation of the hyperelliptic differential, which we call a \emph{contraction datum}.  Note that these tropical differentials come from the target of the admissible cover, and are therefore automatically realisable. We prove the following:
\begin{customthm}{A}[$=$ {Theorem \ref{thm:construction}}]
  Let $(\psi\colon C\to P,\lambda)$ be a log hyperelliptic admissible cover of genus $g$ with a contraction datum. There exists a 
  commutative diagram
  \bcd
C\ar[r,"\sigma"]\ar[d,"\psi"] & \oC\ar[d,"\opsi"]\\
P\ar[r,"\tau"] & \oP
\ecd
such that
  \begin{enumerate}[label=(\roman*)]
   \item  $\tau$ is a contraction to a rational, reduced, Cohen--Macaulay curve $\oP$;
   \item $\oC$ is a (not-necessarily reduced) Gorenstein curve of genus $g$ such that $\sigma^*\omega_{\oC}=\omega_C(\lambda)$;
   \item $\opsi$ is a two-to-one cover, and the quotient of a hyperelliptic involution $\bar\iota\colon\oC\to\oC$. 
  \end{enumerate}
 Moreover, the construction commutes with arbitrary base-change.
 \end{customthm}
Note that constructing the contraction\footnote{A contraction is a surjective morphism of curves $\phi\colon C\to D$ which is an isomorphism outside an exceptional subcurve of $C$, whose connected components are contracted to curve singularities of the same genus in $D$ \cite[\S 2.2]{SmythTowards}. Thus, technically, $\sigma$ is not a contraction when $\oC$ is not reduced.} $\sigma\colon C\to\oC$ directly is problematic: the naive strategy of taking $\on{Proj}(\omega_C(\lambda))$ does not behave well under base-change, and does not in general produce a flat family. One solution has been put forward in \cite{Boz21}, by presenting $\OO_{\oC}$ directly in terms of logarithmic data. Here we pursue a similar strategy, by first contracting the target of the admissible cover---which, being rational, does not present any issues---and then reconstructing $\oC$ as a double cover of the rational, not necessarily Gorenstein curve $\oP$: the cover ``cures'' the failure of $\oP$ to be Gorenstein by building in the structure sheaf all the superabundant differentials. This suggests perhaps that, although restricting to Gorenstein curves is very helpful with the deformation theory of curves, more general Cohen--Macaulay curves may appear quite naturally when looking at covers and other types of maps from curves, see for instance \cite{Heinrich}. 

In \S\ref{sec:local}, we perform a local study of the singularities arising from our construction, including their explicit equations and their dualising bundle.

Next, we proceed to show that our construction is very general: indeed, we have
\begin{customthm}{B}[$=$ {Theorem \ref{thm:converse}}]
    All smoothable, in particular all reduced, hyperelliptic Gorenstein curves arise from the above construction.
\end{customthm}

Finally, in Section \ref{sec:GRC}, we explain the conjectural relation between smoothable differentials and Gorenstein curves. In the presence of a multiscale differential, there is a natural way to log modify the curve $C$, so that a twist of the dualising bundle is trivial on higher levels. This simultaneously avoids non-reduced components arising in the Gorenstein contractions. In the special case of hyperelliptic log differentials, we prove that they always descend to the Gorenstein contractions associated to the cutoffs of their tropicalisations.

\begin{customcor}{C}[$=$ {Proposition \ref{prop:hypdiff}}]
Conjecture \ref{conjG} holds when $C$ is a hyperelliptic curve, and $\eta$ is anti-invariant with respect to the hyperelliptic involution.
\end{customcor}

\subsection*{Conventions} We work throughout over $\on{Spec}(\ZZ[\frac{1}{2}])$.

\subsection*{Acknowledgements}  This project has benefitted from conversations with Dan Abramovich, Francesca Carocci, Dawei Chen, Qile Chen, David Holmes, Scott Mullane, Navid Nabijou, Dhruv Ranganathan, David Smyth, Martin Ulirsch, and Jonathan Wise.

\subsection*{Funding} L.B. has received funding from the Deutsche Forschungsgemeinschaft  (DFG, German Research Foundation) under Germany’s Excellence Strategy EXC-2181/1 - 390900948 (the Heidelberg STRUCTURES Cluster of Excellence) and TRR 326 \emph{Geometry and Arithmetic of Uniformized Structures}, project number 444845124; and from the ERC Advanced Grant SYZYGY of the European Research Council (ERC) under the European Union Horizon 2020 research and innovation program (grant agreement No. 834172).

\section{Logarithmic hyperelliptic curves and differentials}

\subsection{Hyperelliptic admissible covers} Fix a genus $g\geq 2$ and a number of markings $n\geq 0$. Following \cite[Definition 3.5]{Mochizuki}, we will consider families of hyperelliptic admissible covers adapted to the setting of log schemes, together with additional markings at which differentials will later be permitted zeros and poles. We also keep track of the $2g + 2$ points of ramification and branching. Unlike the zeros and poles of differentials, these are not distinguished from each other and should be allowed to exchange in families.

Parallel to \cite[Definition 3.4]{Mochizuki}, we observe that the stack $\mathfrak{M}^{\rm log}_{g,n_1 + n_2}$ of prestable log curves of genus $g$ with $n_1 + n_2$ markings admits a natural action of the symmetric group on $n_1$ letters $\mathfrak{S}_{n_1}$ permuting the first $n_1$ points. The stack quotient $\mathfrak{MS}^{\rm log}_{g,n_1|n_2} \coloneqq \mathfrak{M}^{\rm log}_{g,n_1 + n_2} / \mathfrak{S}_{n_1}$ parametrizes prestable log curves $C$ with a simple divisor $\mathbf{r}$ of $n_1$ ``symmetrized markings" and $n_2$ ordered markings $u_1, \ldots, u_{n_2}$ (collectively $\mathbf{u}$). In order to capture the correct log structure we define such a family to be a pullback of the universal curve of $\mathfrak{MS}^{\rm log}_{g,n_1|n_2}$. We stress that the log structure on the base $S$ is not required to be minimal.

\begin{definition}
A \emph{family of log curves $(C, \mathbf{r}, \mathbf{u})$ of genus $g$ with $n_1$ symmetrized markings and $n_2$ ordered markings} over a log scheme $S$ is a pullback in the category of fs log schemes of the universal curve $\mathfrak{CS}_{g,n_1|n_2}$ of $\mathfrak{MS}^{\rm log}_{g,n_1|n_2}$ along a morphism $S \to \mathfrak{MS}^{\rm log}_{g,n_1|n_2}.$
\end{definition}



\begin{definition}
 A family of $n$-marked \emph{log hyperelliptic admissible covers} of genus $g$ over a log scheme $S$ consists of 
 \begin{enumerate}
   \item families $(C,\mathbf{r},\mathbf{u}) \in \mathfrak{MS}^{\rm log}_{g,2g + 2|2n}(S)$ and $(P,\mathbf{b},\mathbf{v}) \in \mathfrak{MS}^{\rm log}_{0,2g+2|n}(S)$
   \item a finite, fiberwise generically degree 2, Kummer log \'etale morphism $\psi\colon C \to P$ over $S$
 \end{enumerate}
 such that
 \begin{enumerate}
   \item $\psi$ is ramified at $\mathbf{r}$ and possibly some nodes of $C$, and $\psi^{-1}(\mathbf{b}) = 2\mathbf{r}$;
   \item writing $u_{1,1}, u_{1,2}, \ldots, u_{n,1}, u_{n,2}$ for the markings making up $\mathbf{u}$ in $C$ and $v_1, \ldots, v_n$ for the markings making up $\mathbf{v}$ in $P$, we have that $\psi$ maps $u_{i,1}, u_{i,2}$ to $v_i$ for each $i = 1, \ldots, n$.
 \end{enumerate}
\end{definition}

\begin{remark}
More concretely, the map $\psi \colon C \to P$ takes one of the following forms strict \'etale locally in $P$:
\begin{enumerate}
  \item (unramified points) A strict, trivial double cover of a neighborhood of a smooth point, a marked point $v_i$, or a node with the usual log structure;
  \item (ramification over a point of $\mathbf{b}$) The map $\Spec \OO_S[\NN] \to \Spec \OO_S[\NN]$ induced by the multiplication by 2 map from $\NN \to \NN$;
  \item (ramification over a node)
  The map
  \[
    \Spec \OO_S[x,y]/(xy - t) \to \Spec \OO_S[z,w]/(zw - t^2)
  \]
  induced by $z \mapsto x^2, w \mapsto y^2$ and similarly $\log(z) \mapsto 2\log(x), \log(w) \mapsto 2\log(y)$ on log structures.
\end{enumerate}
\end{remark}

Given a family of log hyperelliptic admissible covers over $S$, there is an associated minimal logarithmic structure on $S$, which is a Kummer extension of that of $P$ as a log smooth curve, introducing square roots of the smoothing parameters corresponding to the nodes over which $\psi$ is ramified. Indeed, $\psi$ can be factored as $C\to [C/\iota]=:P^{\rm tw}\to P$, where the first map is strict \'etale, and the second one is Kummer log 
\'etale and birational, albeit it is only representable by DM stacks: $P^{\rm{tw}}$ is an orbicurve with (relative) coarse moduli $P$. 

We say that a log hyperelliptic admissible cover is stable if $(P,\mathbf b,\mathbf v)$ is Deligne--Mumford stable as a rational pointed curve. Notice that $(C,\mathbf r,\mathbf u)$ will be as well. We denote by $\Hcal_{g,n}$ the moduli space of genus $g$, $n$-marked, stable log hyperelliptic admissible covers. Then, $\Hcal_{g,n}$ is represented by a proper DM stack with log structure, which is furthermore (log) smooth \cite[\S 3]{Mochizuki}.

Given a log hyperelliptic admissible cover as above, there is a \emph{hyperelliptic involution} $\iota\colon C\to C$ over $\psi$ fixing $\mathbf r$ and swapping $u_{i,1}$ with $u_{i,2}$.

\subsection{Conewise-linear functions} The tropicalization of a log smooth curve is a family of tropical curves over the base. This is how piecewise-linear geometry and cone complexes enter the picture. As we shall see, they are useful for book-keeping and enrich the algebro-geometric toolbox with a number of combinatorial devices. In particular, conewise-linear (CL) functions on the tropicalization of $C$ are in bijection with global sections $\Gamma(C, \ol{M}_C)$ of the characteristic sheaf on $C$. For us, they will measure the order of vanishing of the logarithmic sections of a line bundle. We refer the reader to \cite{CavalieriChanUlirschWise} for a more careful introduction to these concepts. It is sometimes helpful to consider such CL functions as points of a geometric object.

\begin{definition}[{\cite[Definition 4.1]{MarcusWise}}]
The \emph{tropical multiplicative group} or \emph{tropical line} $\Gtrop$ is the functor on log schemes defined by
\[
  \Gtrop(X) = \Gamma(X, \ol{M}_X^{gp}).
\]
We give the points of $\Gtrop(X)$ a partial order by the rule $\alpha \leq \beta \iff \beta - \alpha \in \Gamma(X, \ol{M}_X)$.

Similarly, the \emph{logarithmic multiplicative group} $\Glog$ is the functor on log schemes defined by
\[
  \Glog(X) = \Gamma(X, M_X^{gp}).
\]
\end{definition}

There is a short exact sequence:
\[0\to\Gm\to\Glog\to\Gtrop\to0.\]
While $\Gtrop$ (and consequently $\Glog$) is not representable by an algebraic stack with log structure, it admits a log smooth cover by log schemes, see \cite[\S 2.2.7]{MolchoWise}. 

\subsection{Log differentials} A \emph{log twisted differential} $\eta$ on $(C,\mathbf u)$ \cite[\S 3]{Chen2} is a logarithmic section of (the total space of) the dualising sheaf $\omega_{C/S}$, where the latter is endowed with the log structure generated by the zero section and the pullback of the log structure of $C$.
\begin{remark}\label{rmk:twisted}
    A log twisted differential is equivalently a section of the fibre product:
    \bcd[ampersand replacement=\&]
\& S \ar[d,"{(C,\omega_{C/S})}"] \\
\mathbf{Div}\ar[r,"{\rm{aj}}"] \& {\mathcal{P}ic}
    \ecd
where the bottom arrow is defined in \cite[Definitions 4.2.1 and 4.3.1]{MarcusWise}, or a section of the log line bundle associated to the $\Gm$-torsor underlying $\omega_{C/S}$ under the extension $\Gm\to\Glog$.
\end{remark}
Even when $C$ is smooth, $C$ must be endowed with non-trivial log structure at the zeroes of $\eta$, which we henceforth take to be (a subset of) the markings $\mathbf u$ above (independent of $C$ being smooth or nodal). The multiplicity $m_i$ of $u_i$ as a zero of the differential is encoded in the logarithmic enhancement of the $\eta$, and is locally constant in families; thus, it makes sense to fix a non-negative integer partition $\mu=(m_1,\ldots,m_n)$ of $2g-2$ from the beginning. Moreover, the tropicalization of $\eta$ is a conewise-linear function $\blambda$ on the tropicalization $\Gamma$ of $C$, measuring the order of vanishing of the differential at the generic point of each component of $C$. 

Components of $C$ on which $\eta$ vanishes are sometimes referred to as \emph{degenerate}.
Now assume for simplicity that $S$ is a log point. Importantly, the log differential $\eta$ also contains the information of a meromorphic differential $\eta_v$ on every component $C_v$ of $C$---even the degenerate ones---whose order of zero/pole at the special points of $C_v$ is governed by $\blambda$ (see \cite[\S 3.2]{Chen2}). We stress that the order of zero/pole of $\eta_v$ on either side of a node differs from the slope of $\blambda$ by $1$, due to the fact that $\omega_C$ restricts on $C_v$ to the sheaf of differentials with logarithmic poles at the nodes. We can indeed write:
\begin{equation}\label{eqn:twisted_diff}
 \omega_C=\OO_C(\blambda).
\end{equation}

Notice that here $\blambda$ has slope $m_i$ on the leg corresponding to $u_i$; had we set this slope to $0$, we could have written:
\[\omega_C(-\sum m_iu_i)=\OO_C(\blambda_{\rm vert}).\]

\subsection{Hyperelliptic differentials} We recall \cite[Definition 5.3]{Chen2}.
\begin{definition}
A \emph{log twisted hyperelliptic differential} over a log scheme $S$ is the datum of
\begin{enumerate}[label=(\roman*)]
 \item a log hyperelliptic admissible cover $\psi\colon(C,\mathbf r,\mathbf u)\to(P,\mathbf b,\mathbf v)$ over $S$, and
 \item a log twisted differential $\eta$ on $C$, such that
\end{enumerate}
 \begin{equation}\label{eqn:involution}
  \iota^*(\eta)=-\eta,
 \end{equation}
where $\iota$ is the hyperelliptic involution. 
\end{definition}

\begin{remark}\label{rmk:antiinvariant1}
    Let $C$ be a smooth hyperelliptic curve. Then, every (global holomorphic) differential on $C$ is $\iota$-anti-invariant. This follows, for instance, from the well-known fact that if an affine patch of $C$ is written as $\{y^2=p(x)\}\subseteq\Aaff^2$ (with $p$ a square-free polynomial of degree $2g+2$), then $\iota$ acts as $(x,y)\mapsto(x,-y)$, and a basis of the space of differentials on $C$ is given by $\{\frac{\on{d}\!x}{y},\ldots,x^{g-1}\frac{\on{d}\!x}{y}\}$ (compatibly with the Riemann--Hurwitz formula $\omega_C=\psi^*\omega_P(\frac{1}{2}\mathbf b)$). Alternatively, any $\iota$-invariant differential descends to $\PP^1$, and must therefore be trivial. We will see a generalisation of this statement in Remark \ref{rmk:antiinvariant2}. It follows that limits of Abelian differentials on smooth elliptic curves are anti-invariant log twisted differentials.
\end{remark}

Given the latter condition, the vanishing orders of $\eta$ at the conjugate points $u_{i,1}$ and $u_{i,2}$ are the same. We may therefore denote the vanishing order of $\eta$ by an $n$-tuple $\mu=(m_1,\ldots,m_n)\in \NN^n_{g-1}$ of non-negative integers summing to $g-1$. We restrict our attention to strata of hyperelliptic differentials with no zeroes at the Weierstrass points.\footnote{Differentials with zeroes (of even multiplicity $2m'$) at a Weierstrass point arise in the boundary of these spaces when a marking $v$ of contact order $m'$ moves off to a trivial bubble together with a single branch point (the node is then also branching).}

\subsection{Balancing} We now give a description of the slopes of the tropicalization $\blambda$ of $\eta$. We will use this later to describe contraction data in terms of CL functions that ``locally look like" $-\blambda.$ The upshot is that, in the hyperelliptic case, the triviality of an appropriate twist of the canonical bundle can be measured numerically.

Note that condition \eqref{eqn:involution} implies that the tropicalization $\blambda$ descends to the quotient of the tropicalization $\Gamma$ of $C$ by $\on{trop}(\iota)$, that we denote by $\Pcal$. It is the tropicalization of the orbicurve $P^{\rm tw}:=[C/\iota]$, whose coarse moduli (relative to $S$) is $P$. Thus, $\Pcal$ admits a morphism to the tropicalization $\Tau$ of $P$, which is a Kummer extension of the edges corresponding to branching nodes of $P$. 
Let $\blambda_T$ denote the CL function on $\Pcal$ pulling back to $\blambda$ on $\Gamma$. Very concretely, the previous discussion implies that we may view $\blambda_T$ as a CL function on $T$, except that it may have half-integral slopes on the branching edges. Equation \eqref{eqn:twisted_diff} descends then to:
\begin{equation}\label{eqn:twisted_diff_P}
 \omega_{P^{\rm tw}}=\OO_{P^{\rm tw}}(\blambda_T).
\end{equation}
Indeed, $C\to P^{\rm tw}$ is \'etale, whereas $\rho\colon P^{\rm tw}\to P$ satisfies (compare with \cite[Proposition 2.5.1]{Chiodo}) $\omega_{P^{\rm tw}}=\rho^*\omega_P(\frac{1}{2}\mathbf b)$. Since $\rho$ induces an isomorphism of Picard groups up to torsion, and since $P$ is rational, (a multiple of) condition \eqref{eqn:twisted_diff_P} can be checked numerically, and becomes:
\begin{equation}\label{eqn:balancing}
 \val(v)-2+\frac{1}{2}\deg(\mathbf b)(v)-\on{div}(\blambda_T)(v)=0,
\end{equation}
for every vertex $v$ of $\Tau$. Here, $\val$ denotes the edge valency, $\deg(\mathbf b)$ the multi-degree (or tropicalization) of the branch divisor, and $\on{div}(\blambda_T)(v)$ the sum of the outgoing slopes of $\blambda_T$ along all the edges adjacent to $v$. We may as well modify $\blambda_T$ to $\tilde\lambda_T$ by making its slope $-\frac{1}{2}$ on all the legs corresponding to $\mathbf b$, so that the previous equation now reads:
\begin{equation}\label{eqn:balancing_tilde}
 \val(v)-2=\on{div}(\tilde\lambda_T)(v).
\end{equation}

\subsection{Rubber differentals and level graphs}\label{sec:levels} 
For our application to the moduli space of differentials, we are going to require the following technical adaptation.
If we impose that $\blambda$ (or, what is the same, $\blambda_T$, $\blambda_{\rm vert}$, or $\tilde\lambda_T$) as a map from $\Gamma$ to $\Gtrop$  has totally ordered image, and that the pullback of the corresponding subdivision of $\Gtrop$ to $C$ still is a log curve,\footnote{This amounts to requiring that $\eta\colon C\to\omega_C$ is a logarithmic stable map in the sense of Kim \cite{Kim}.} this allows us to identify $\eta$ with a \emph{logarithmic rubber differential} in the sense of \cite[Definition 2.7]{Tale}, or equivalently, according to the same paper, a \emph{generalised multiscale differential} \cite{BCGGM}.
\begin{remark}
In the language of \cite{MarcusWise}, this is equivalent to replacing $\mathbf{Div}$ with $\mathbf{Rub}$ in the diagram of Remark \ref{rmk:twisted}.
\end{remark}

We can now use $-\blambda$ (resp. $-\blambda_T$) to make $\Gamma$ (resp. $\Tau$) into an \emph{enhanced level graph}, as explained in \cite[\S 5.1]{Tale}, translating it in particular so that its maximum value is $0$. We denote by $X_{\boxtimes i}$ the (possibly disconnected) subcurve of $X$ at level $\boxtimes i$ ($i=0,\ldots,-N$), where $X$ is either $C,\Gamma,P,$ or $\Tau$, and $\boxtimes$ is either $<,\leq,=,\geq,$ or $>$; we write (notice the sign change!) \[\lambda_i=\max\{-\blambda-i,0\},\]
and similarly for $\blambda_T$. Notice that $\lambda_i$ is in the tropical canonical linear series of $\Gamma_{\geq i}$, and in fact $\omega_C(\lambda_i)$ is trivial on $C_{>i}$. In the hyperelliptic case, this can be deduced from the analogous statement for $P$.

In the following we extract the combinatorial content of a tropical hyperelliptic differential; this is all we need in the next section to construct a Gorenstein contraction of a log hyperelliptic admissible cover.
\begin{definition}
 Let $\psi\colon(C,\mathbf r,\mathbf u)\to(P,\mathbf b,\mathbf v)$ be a log hyperelliptic admissible cover. A \emph{contraction datum} is a CL function $\lambda_T$ on $\Pcal=\on{trop}([C/\iota])$ such that $\lambda_T(v)\geq 0$ for every vertex $v$ of $\Pcal$, and the balancing equation
 \begin{equation}\label{eqn:balancing_neg}
 \val(v)-2+\frac{1}{2}\deg(\mathbf b)(v)+\on{div}(\lambda_T)(v)=0,
\end{equation}
holds for every vertex $v$ of $\on{Supp}(\lambda_T)=\Pcal_{>0}$.
\end{definition}

\begin{remark}
Equation \eqref{eqn:balancing_neg} is obtained by negating the divisor of $\blambda_T$ in Equation \eqref{eqn:balancing}.
\end{remark}

A log rubber differential on $C\to S$ induces a minimal logarithmic structure on $S$: $\blambda_{\rm vert}$ induces a subdivision of $\Gtrop$ along the image of the vertices of $\Gamma$, and the tropical parameters are fractions of the lengths of the finite edges in this subdivision. A log differential is called stable if $C$ is. We denote the moduli space of stable log rubber  differentials by $\Hcal^\dagger(\mu)$, as in \cite{Chen2}. It is represented by a separated DM stack with locally free log structure, and is finite (in particular, representable) over the Hodge bundle over $\mathbf{Rub}$ (c.f. \cite[Corollary 3.2]{Chen2}), which is itself logarithmically \'etale over $\Mbar_{g,n}$ \cite[Proposition 5.1.2]{MarcusWise}. Typically, $\Hcal^\dagger(\mu)$ is not irreducible.

\begin{remark}
 Asking values of $\lambda_T$ to be comparable to $0$ can be thought of as taking a subdivision of the moduli space in the logarithmic sense. Various versions of this (e.g. requesting all values to be comparable) may produce less economic moduli spaces, that may though be better behaved from the point of view of deformation theory (cf. \cite[\S4.6]{RSPW1}).
\end{remark}

\begin{example}
To illustrate the combinatorics of contraction data, we consider the dual graphs of several log hyperelliptic admissible covers with $g = 2$ and $n = 1$, i.e. two simple zeroes at conjugate points. We first consider
the possible stable hyperelliptic admissible covers where $\Pcal$ has a single edge. We may attempt to construct a contraction datum supported on a single vertex $v$ by computing the required slope of $\lambda_T$ at $v$ using Equation \eqref{eqn:balancing_neg}, then shifting $\lambda_T$ so that its minimum value is 0. The possibilities are illustrated in Figure \ref{fig:contraction_datum_example_1}. We indicate the number of legs coming from $\mathbf{b}$ and the ramified edges (nodes) in blue, the legs coming from markings in black, and the positive slopes of $\lambda$, $\lambda_T$ in green. We recall that pulling back $\lambda_T$ to $\lambda$ doubles slopes on ramified edges.
 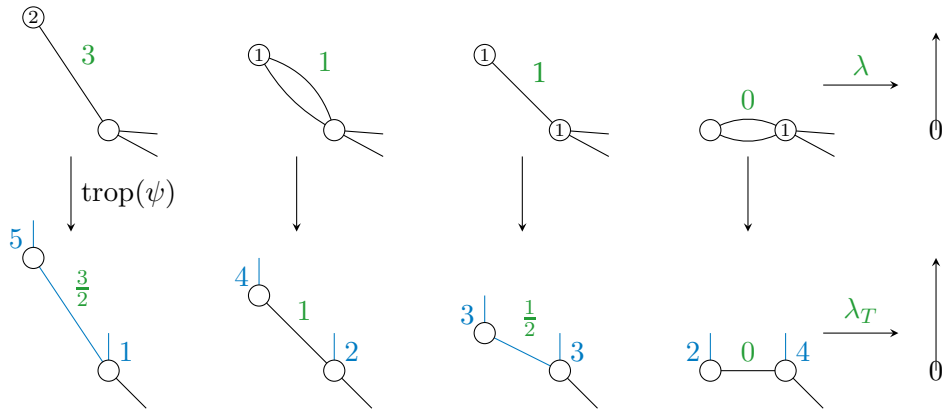
\begin{figure}[h]
  \begin{tikzpicture}

  \draw (0,4.7)--node[above right]{\color{sebgreen}{$3$}}(1,3.2);
  \draw (1,3.2)--(1.65,2.85) (1,3.2)--(1.65,3.15);
  \draw[fill=white] (0,4.7) circle (4pt) (1,3.2) circle (4pt);
  \node at (0,4.7) {\tiny$2$};

  \draw[-stealth] (0.5,2.85)--node[right]{$\on{trop}(\psi)$} (0.5,1.85);

  \draw[sebblue] (0,1.5)--(1,0);
  \draw (1,0)--(1.5,-0.5); 
  \draw[sebblue] (0,1.5)--node[left]{$5$}(0,2);
  \draw[sebblue] (1,0)--node[right]{$1$}(1,0.5);
  \draw[fill=white] (0,1.5) circle (4pt) (1,0) circle (4pt);
  \node at (0.65,1.1) {\color{sebgreen}{$\frac{3}{2}$}};


  \draw (3,4.2) to[out=-15,in=105] node[above right]{\color{sebgreen}{$1$}} (4,3.2);
  \draw (3,4.2) to[out=-60,in=150] (4,3.2);
  \draw (4,3.2)--(4.65,2.85) (4,3.2)--(4.65,3.15);
  \draw[fill=white] (3,4.2) circle (4pt) (4,3.2) circle (4pt);
  \node at (3,4.2) {\tiny$1$};

  \draw[-stealth] (3.5,2.85)-- (3.5,1.85);

  \draw (3,1)--(4,0);
  \draw (4,0)--(4.5,-0.5); 
  \draw[sebblue] (3,1)--node[left]{$4$}(3,1.5);
  \draw[sebblue] (4,0)--node[right]{$2$}(4,0.5);
  \draw[fill=white] (3,1) circle (4pt) (4,0) circle (4pt);
  \node at (3.6,.8) {\color{sebgreen}{$1$}};


  \draw (6,4.2)--node[above right]{\color{sebgreen}{1}} (7, 3.2);
  \draw (7,3.2)--(7.65,2.85) (7,3.2)--(7.65,3.15);
  \draw[fill=white] (6,4.2) circle (4pt) (7,3.2) circle (4pt);
  \node at (6,4.2) {\tiny$1$};
  \node at (7,3.2) {\tiny$1$};

  \draw[-stealth] (6.5,2.85)-- (6.5,1.85);

  \draw[sebblue] (6,0.5)--(7,0);
  \draw (7,0)--(7.5,-0.5); 
  \draw[sebblue] (6,0.5)--node[left]{$3$}(6,1);
  \draw[sebblue] (7,0)--node[right]{$3$}(7,0.5);
  \draw[fill=white] (6,0.5) circle (4pt) (7,0) circle (4pt);
  \node at (6.6,.7) {\color{sebgreen}{$\frac{1}{2}$}};


  \draw (9,3.2) to[out=30,in=150] node[above]{\color{sebgreen}{$0$}} (10, 3.2);
  \draw (9,3.2) to[out=-30,in=-150]  (10, 3.2);
  \draw (10,3.2)--(10.65,2.85) (10,3.2)--(10.65,3.15);
  \draw[fill=white] (9,3.2) circle (4pt) (10,3.2) circle (4pt);
  \node at (10,3.2) {\tiny$1$};

  \draw[-stealth] (9.5,2.85)-- (9.5,1.85);

  \draw (9,0)--node[above]{\color{sebgreen}{$0$}} (10,0);
  \draw (10,0)--(10.5,-0.5); 
  \draw[sebblue] (9,0)--node[left]{$2$}(9,0.5);
  \draw[sebblue] (10,0)--node[right]{$4$}(10,0.5);
  \draw[fill=white] (9,0) circle (4pt) (10,0) circle (4pt);

  \draw[-stealth] (10.5,.5)--node[above]{\color{sebgreen}{$\lambda_T$}} (11.5,.5);

  \draw[-stealth] (12,0)node{$0$}--(12,1.5);

  \draw[-stealth] (10.5,3.8)--node[above]{\color{sebgreen}{$\lambda$}} (11.5,3.8);

  \draw[-stealth] (12,3.2)node{$0$}--(12,4.5);
  \end{tikzpicture}

  \caption{Contraction data on hyperelliptic admissible covers with one edge.}
  \label{fig:contraction_datum_example_1}
 \end{figure}
Since Equation \eqref{eqn:balancing_neg} is stable under generization, we can deduce the slopes of contraction data on genus two admissible covers
with more components by generizing to the single-edge graphs. See Figures \ref{fig:contraction_datum_example_1} and \ref{fig:contraction_datum_example_2}.

\begin{figure}[h]
  \begin{tikzpicture}
  \draw (0,5.5) to[out=30,in=150]node[above]{\color{sebgreen}{$0$}} (1,5.5)
  (0,5.5) to[out=-30,in=-150] (1,5.5)
  (1,5.5)--node[above right]{\color{sebgreen}{$1$}} (2,4.5)
  (2,4.5) to[out=-15,in=105] node[above right]{\color{sebgreen}{$1$}} (3,3.5)
  (2,4.5) to[out=-75,in=165] (3,3.5) --node[above]{\color{sebgreen}{$0$}} (4,3.5);
  \draw (4,3.5)--(4.5,3) (4,3.5)--(4.65,3.3);
  \draw[fill=white] (0,5.5) circle (4pt) (1,5.5) circle (4pt) (2,4.5) circle (4pt) (3,3.5) circle (4pt) (4,3.5) circle (4pt);

  \draw[-stealth] (2,3)--node[right]{$\on{trop}(\psi)$}(2,2);

  \draw (0,1.5) --node[above]{\color{sebgreen}{$0$}} (1,1.5) (2,1) --node[above right]{\color{sebgreen}{$1$}} (3,0);
  \draw[sebblue] (0,1.5)--node[left]{$2$} (0,2)  (1,1.5)--(1,2) (2,1)--(2,1.5) (3,0)--(3,0.5) (4,0)--(4,0.5);
  \draw[sebblue] (1,1.5) --node[above right]{\color{sebgreen}{$\frac{1}{2}$}}(2,1)  (3,0) --node[above]{\color{sebgreen}{$0$}} (4,0);
  \draw (4,0)--(4.5,-0.5);
  \draw[fill=white] (0,1.5) circle (4pt) (1,1.5) circle (4pt) (2,1) circle (4pt) (3,0) circle (4pt) (4,0) circle (4pt);
  \end{tikzpicture}

  \caption{A contraction datum on a larger log hyperelliptic admissible cover supported on the left 3 vertices. Unnumbered blue legs are single branch legs.}
  \label{fig:contraction_datum_example_2}
 \end{figure}
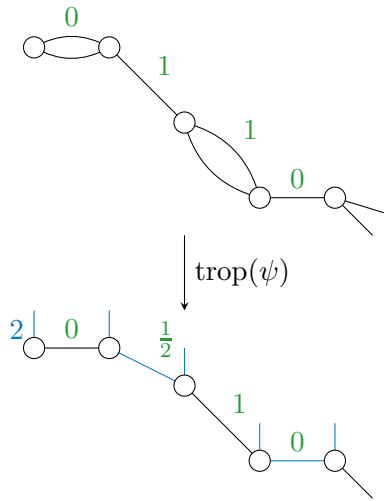

\end{example}

\section{Construction of Gorenstein contractions}\label{sec:construction}

Let now $(\psi\colon C\to P,\lambda_T)$ be a log hyperelliptic admissible cover over $S$ with a contraction datum.
Using $\lambda_T$, we will produce a hyperelliptic Gorenstein contraction $\oC$ of $C$. The strategy is to first contract $P$ to $\oP$, which is easier since $P$ is rational, and then construct $\oC$ as a double cover of $\oP$, by twisting the (horizontal) branch divisor $\mathbf b$ of $\psi$ with the (vertical) CL function $\lambda_T$. The dualising sheaf and its twists are seen to play a key role in connection with the Gorenstein condition. We stress that the double cover will not be flat whenever the dualising sheaf of $\oP$ is not a line bundle, and its branching locus will only be a generalised divisor.

\subsection{Contracting the rational curve} Consider the line bundle $\Lcal:=\omega_{P^{\rm tw}}(\lambda_T)$. 
Note that $\Lcal^{\otimes 2}$ descends to a line bundle $L^{\otimes2}=\omega_P^{\otimes2}(\mathbf b+2\lambda_T)$ of degree $2g-2$ and non-negative multi-degree on $P$. In particular, since $P$ is rational, $L$ has vanishing higher cohomology, $\pi_{P,*}L^{\otimes 2k}$ is a vector bundle of rank $2k(g-1)+1$ on $S$, and $L^{\otimes 2}$ is generated by global sections. Moreover, the formation of $\pi_{P,*}L^{\otimes 2k}$ commutes with arbitrary base-change by \cite[Theorem III.12.11]{HartshorneAG}. 
Let
\[P\xrightarrow{\tau}\oP:=\underline{\on{Proj}}_S\left(\bigoplus_{k\geq 0}\pi_{P,*}L^{\otimes 2k}\right)\]
be the resulting contraction. The fibres of $\tau$ are connected components of the locus where $L^{\otimes2}$ is trivial (i.e. of multi-degree $\underline 0$). In particular they are connected, rational, prestable curves. The map $\OO_P\to\OO_{\oP}$ induces an identification of the latter with $\on{R}\!\tau_*\OO_P$, and $\oP$ is a family of rational Cohen--Macaulay curves. In particular, the singularities appearing in the fibres of $\oP$ are ordinary $m$-fold points; they are Gorenstein (even l.c.i.) if and only if $m\leq2$. Let $\OO_{\oP}({\mathbbm 2})$ denote the line bundle on $\oP$ induced by $L^{\otimes2}$ on $T$.

\subsection{Square roots} 
We will need a square root of $\OO_{\oP}({\mathbbm 2})$ for our construction. Such a line bundle does not always exist on $\oP$ itself, but
after reintroduction of some stack structure on $P$ and $\oP$, we can find such a square root.

\begin{lemma} (c.f. \cite[Lemma 3.5]{Fedorchuk})
There is a commutative diagram
\[
\begin{tikzcd}
  P' \ar[d] \ar[r, "\tau'"] & \oP' \ar[d] \\
  P \ar[r, "\tau"] & \oP
\end{tikzcd}
\]
and unique line bundles $L' \in \Pic(P')$, $\OO_{\oPcal}({\mathbbm 1}) \in \Pic(\oP')$ with the following properties:
\begin{enumerate}
  \item $P'$ is a partial coarsening of $P^{\rm tw}$, while $\tau'$ is representable;
  \item $\OO_{\oP'}({\mathbbm 1})^{\otimes 2}$ is isomorphic to the pullback of $\OO_{\oP}({\mathbbm 2})$,
  \item $\tau^*\OO_{\oP'}({\mathbbm 1}) \cong L'$, and so $(L')^{\otimes 2}$ is isomorphic to the pullback of $L^{\otimes 2}$;
   \item $v_i^*(L'|_{P^{\rm tw}}) \cong v_i^*\Lcal$ for each $i = 1, \ldots, n$.
\end{enumerate}
\end{lemma}
\begin{proof}
Given a node $p$ of $P$, we say that $p$ is odd (resp. even) for a line bundle $B$ of even degree if $B$ restricts to an odd (resp. even) degree line bundle on the connected components of the normalization of $P$ at $p$. Let $Z_{odd}$ be the locus of odd nodes for $\mathcal{L}^{\otimes 2}$ outside of the support of $\lambda_T$. Similarly, let $Z_{even}$ be the locus of even nodes for $\mathcal{L}^{\otimes 2}$ unioned with the support of $\lambda_T$. We may construct the partial coarsening $P'$ by gluing $P$ away from $Z_{odd}$ with $P^{\rm tw}$ away from $Z_{even}$. Using the fact that $\tau$ is an isomorphism away from its exceptional locus, we construct the twisted curve $\oP'$ by gluing $P'$ away from the exceptional locus of $\tau$ with $\oP$ away from the image of $Z_{odd}$.

We begin by considering the case of an individual curve $P$. Imagine first trying to find a square root of $L^{\otimes 2}$ on $P$; since $P$ is rational, this is purely a matter of multidegree. Since $\omega_P^{\otimes 2}$ has even degree on every component, a problem may occur only when there is a node $p$ which is odd for $\mathbf{b}$, i.e. a node of $P$ separating $\mathbf{b}$ into two odd parts. The twisting of $P'$ at some of the odd nodes will resolve the issue, and the same twisting will guarantee the existence of a square root of $\OO_{\oP}(\mathbbm2)$ on $\oP$. Importantly, the odd nodes in the exceptional locus of $\tau$ need no twisting because $\lambda_T$ acts as a correction factor.

Indeed, if $\lambda_T$ has non-zero slope along the edge $e_p$ corresponding to $p$, then at least one of the two adjacent vertices is contained in the support of $\lambda_T$; call it $v$. Notice that Equation \eqref{eqn:balancing_neg} is stable under edge contractions, so in applying it to $v$ we may as well assume that $e_p$ is the only edge of $T$. Then we see from Equation \eqref{eqn:balancing_neg} that $\lambda_T$ must have half-integral slope along $e_p$ so that it can balance $\mathbf{b}$. It follows that $p$ is not odd for $L^{\otimes2}$, and neither is its image in $\oP$ for $\OO_{\oP}(\mathbbm2)$.

Returning to the case of an arbitrary family, choose $v_i : S \to P^{\rm tw}$ to be any of the marked points. As in \cite[Lemma 3.6]{Fedorchuk}, a standard descent argument shows that the square roots on
fibers can be glued to a unique line bundle $L'$ so that $(L')^{\otimes 2}$ is isomorphic to the pullback of $L^{\otimes 2}$ and $v_i^*(L'|_{P^{\rm tw}}) \cong v_i^*\Lcal$. Since $P^{\rm tw} \to P$ is an isomorphism in the complement of the odd nodes, the rigidification condition $v_i^*(L'|_{P^{\rm tw}}) \cong v_i^*\Lcal$ assures that $L'$ is isomorphic to $\Lcal$ on the complement of the odd nodes, so the analogous rigidification conditions $v_j^*(L'|_{P^{\rm tw}}) \cong v_j^*\Lcal$ hold as well.

Moreover, $L'$ is trivial on the exceptional locus of $\tau$, and it therefore descends to a line bundle $\OO_{\oP'}(\mathbbm1)$ which squares to (the pullback of) $\OO_{\oP}(\mathbbm2)$.
\end{proof}

We will henceforth abuse notation by suppressing the $'$. Since $\tau$ is an isomorphism around the twisting nodes, it will be apparent that twisting does not affect the rest of our construction significantly.

\subsection{The double cover} We will construct a Gorenstein double cover $\opsi\colon\oC\to\oP$ making the following diagram of curve-line bundle pairs commutative:
\begin{equation}\label{diag:contraction}
\begin{tikzcd}
(C,\omega_C(\lambda))\ar[r,"\sigma"]\ar[d,"\psi"] & (\oC,\omega_{\oC})\ar[d,"\opsi"]\\
(P,L)\ar[r,"\tau"] & (\oP,\OO_{\oP}(\mathbbm 1))
\end{tikzcd}
\end{equation}

For this we let (abusively omitting $\bar\psi_*$)
\[\OO_{\oC}:=\OO_{\oP}\oplus\left(\omega_{\oP}\otimes\OO_{\oP}(-\mathbbm 1)\right).\]

Let us denote $\omega_{\oP}\otimes\OO_{\oP}(-\mathbbm 1)$ by $M$. Notice that this is a rank one, torsion-free (i.e. depth one) sheaf, which fails to be a line bundle wherever $\oP$ has worse-than-nodal singularities. As a consequence, $\opsi$ will fail to be flat at those points. On the other hand, since $\oP$ is flat over the base, then so is $\oC$.

\begin{remark}
 Twisting the dualising sheaf of a non-Gorenstein curve with a line bundle is known to produce another irreducible component of the compactified Picard scheme \cite{Kass}.
\end{remark}

\subsection{}\label{sec:algebra_structure} In order to give $\OO_{\oC}$ an $\OO_{\oP}$-algebra structure, we need a cosection 
$M^{\otimes2}\to\OO_{\oP}$.
 Since $\on{R}\!\tau_*\OO_P=\OO_{\oP}$, a direct application of Grothendieck duality yields (suppressing $\on{R}$) 
 \[\omega_{\oP}=\mathcal{H}om(\tau_*\OO_P,\omega_{\oP})=\tau_*\mathcal{H}om(\OO_P,\tau^!\omega_{\oP})=\tau_*\omega_P.\]
 
 Since $\lambda_T\geq 0$, the log structure of $P$ includes a (generalised) effective Cartier divisor $\OO_P(-\lambda_T)\to\OO_P$. This, together with $\mathbf{b}$ (and adjunction), gives us a map 
 \begin{equation*}\label{eqn:branching_on_P}
 \tau^*\omega_{\oP}^{\otimes 2}\to\tau^*\tau_*\omega_P^{\otimes2}\to\omega_P^{\otimes2}\to\omega_P^{\otimes2}(2\lambda_T+\mathbf b)=L^{\otimes 2} 
 \end{equation*}
 pushing forward to the desired
 \[\omega_{\oP}^{\otimes 2}\to\tau_*\tau^*\omega_{\oP}^{\otimes 2}\to\OO_{\oP}(\mathbbm 2).\]

Note that there is an obvious involution $\bar\iota$ of $\OO_C$ over $\OO_T$, acting as $-1$ on sections of $M$. 
Also, note that (the fibres of) $\oC$ fail to be reduced whenever there is a component $P_1$ of $P$ in the support of $\lambda_T$ such that the degree of $L$ on $P_1$ is positive.

\subsection{} We argue that $\oC$ is Gorenstein. More precisely, its dualising sheaf $\omega_{\oC}$ can be identified with the line bundle $\opsi^*\OO_{\oP}(\mathbbm 1)$. Duality for the finite morphism $\opsi$ gives
 \[\opsi_*\omega_{\oC}=\mathcal{H}om(\psi_*\OO_{\oP},\omega_{\oP})=\omega_{\oP}\oplus\OO_{\oP}(\mathbbm 1)\otimes\mathcal{H}om(\omega_{\oP},\omega_{\oP}).\]
 Since $\OO_{\oP}\to \mathcal{H}om(\omega_{\oP},\omega_{\oP})$ is an isomorphism by \cite[Corollary 1.7]{HarGenDiv}, we get a morphism $\OO_{\oP}(\mathbbm 1)\to \opsi_*\omega_{\oC}$, or equivalently $\opsi^*\OO_{\oP}(\mathbbm 1)\to \omega_{\oC}$. Since $\opsi_*$ is exact, it is enough to show that this is an isomorphism after pushforward, which follows from push-pull and the above:
 \[\opsi_*\opsi^*\OO_{\oP}(\mathbbm 1)=\OO_{\oP}(\mathbbm 1)\oplus\omega_{\oP}=\opsi_*\omega_{\oC}.\]

\subsection{The contraction} We argue that there exists a morphism $\sigma\colon C\to\oC$ covering $\tau$. Since $\opsi$ is affine, it is enough to define an $\OO_S$-algebra map $\OO_{\oC}\to\sigma_*\OO_C$ after pushing forward along $\opsi$. Since we want $\opsi_*\sigma_*=\tau_*\psi_*$ by adjunction it is enough to define a morphism $\psi^*\tau^*\opsi_*\OO_{\oC}\to\OO_C$. We focus on $M$ since the pullback of $\OO_{\oP}$ is naturally identified with $\OO_C$. The map is induced by $\tau^*\omega_{\oP}\to\omega_P$ (see \ref{sec:algebra_structure}) and by the effective Cartier divisor $\mathbf r+\lambda$ on $C$ as follows:
\begin{align*}
\psi^*\tau^*\omega_{\oP}(-\mathbbm 1)=\psi^*(\tau^*\omega_{\oP}\otimes L^{-1})=\psi^*\tau^*\omega_{\oP}\otimes(\omega_C(\lambda))^{-1}\to \\ \psi^*\omega_P\otimes(\omega_C(\lambda))^{-1}=\omega_C(-\mathbf r)\otimes(\omega_C(\lambda))^{-1}=\OO_C(-\mathbf r-\lambda)\to\OO_C.
\end{align*}

\subsection{} Finally, we note that the arithmetic genus of $\oC$ is the same as that of $C$. This follows from smoothing and the following important observation that the construction of Diagram \eqref{diag:contraction} commutes with arbitrary base-change.

 The arithmetic genus of $\oC$ can also be computed directly. Since $\opsi$ is affine, it is enough to compute $h^0(\oP,\opsi_*\omega_{\oC})=h^0(\oP,\OO_{\oP}(\mathbbm 1)\oplus\omega_{\oP})$. Since $\oP$ is rational, $h^0(\oP,\omega_{\oP})=0$. On the other hand, $h^0(\oP,\OO_{\oP}(\mathbbm 1))=h^0(P,L)=g$, since $L$ is a non-negative line bundle of total degree $g-1$ on the rational nodal curve $P$.

 Summing up, we have proved the following
 \begin{theorem}\label{thm:construction}
  Let $(\psi\colon C\to P,\lambda_T)$ be a log hyperelliptic admissible cover of genus $g$ with a contraction datum. There exists a contraction $(\sigma,\tau)$ to $(\opsi\colon\oC\to\oP)$ such that
  \begin{enumerate}[label=(\roman*)]
   \item $\oP$ is a rational, reduced, Cohen--Macaulay curve;
   \item $\oC$ is a (not-necessarily reduced) Gorenstein curve of genus $g$ such that $\sigma^*\omega_{\oC}=\omega_C(\lambda)$;
   \item $\opsi$ is the quotient of a hyperelliptic involution $\iota\colon\oC\to\oC$.
  \end{enumerate}
 Moreover, the construction commutes with arbitrary base-change.
 \end{theorem}

\subsection{Examples} We are going to work out some familiar examples of curve singularities of low genus, before providing a general description of the local equations of the singularities just constructed.

\begin{example}
 We construct a genus one singularity with six branches, c.f. \cite{Smyth,Boz21}.
 \begin{figure}[h]
  \begin{tikzpicture}

   \draw (-1.1,0)--(0,1)--(-.8,0) (-.15,0)--(0,1)--(.15,0) (.8,0)--(0,1)--(1.1,0);
   \foreach \i in {(-1.1,0),(0,1),(-.8,0),(-.15,0),(.15,0), (.8,0),(1.1,0)} {\draw[fill=white]  \i circle (4pt);}
   \node at (0,1) {\tiny$1$};
   \draw[-stealth] (1.5,.5)--node[above]{$\on{trop}(\psi)$} (3,.5);
  \draw (3.5,0)--node[right]{\color{sebgreen}{$1$}}(4.5,1)--node[right]{\color{sebgreen}{$1$}}(5.5,0) (4.5,0)--node[right]{\color{sebgreen}{$1$}}(4.5,1);
  \draw[sebblue] (4.5,1)--node[right]{$4$}(4.5,1.5);
  \draw[fill=white] (3.5,0) circle (4pt) (4.5,1) circle (4pt) (5.5,0) circle (4pt) (4.5,0) circle (4pt);
  \draw[-stealth] (6,.5)--node[above]{\color{sebgreen}{$\lambda_T$}} (7.5,.5);

  \draw[-stealth] (8,0)node{$0$}--(8,1.5);
  \end{tikzpicture}
 \end{figure}
 
In blue we note the number of branch points.  The line bundle $L$ is trivial on the central vertex of the tree, so the corresponding component is contracted into an ordinary (rational) $3$-fold point. The sheaf $\omega_{\oP}\otimes\OO_{\oP}(-1)$ has two generators $u_{12}$ and $u_{13}$, corresponding to the generators $(\frac{\on{d}s_1}{s_1},-\frac{\on{d}s_2}{s_2},0)$ and $(\frac{\on{d}s_1}{s_1},0,-\frac{\on{d}s_3}{s_3})$. The resulting singularity has thus local equations given by:
\[\hat\OO_{\oC,q}=\kfield[\![s_1,s_2,s_3]\!][u_{12},u_{13}]/(s_1s_2,s_1s_3,s_2s_3,u_{12}^2-s_1^2-s_2^2,u_{13}^2-s_1^2-s_3^2),\]
which is isomorphic to $\kfield[\![x_1,\bar x_1,x_2,\bar x_2, x_3]\!]/I_6$ \cite[Proposition A.3]{Smyth} via \[s_1=x_1-\bar x_1,s_2=x_2-\bar x_2,s_3=\frac{1}{2}x_3,u_{12}=x_1+\bar x_1-(x_2+\bar x_2), u_{13}=x_1+\bar x_1-x_3.\]
\end{example}

\begin{example}
 We construct a genus two singularity ``of type $I$'' with three branches, c.f. \cite{Bat19}.
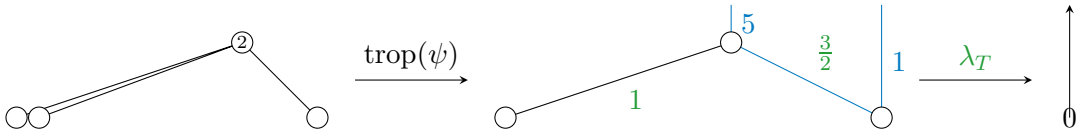
 \begin{figure}[h]
  \begin{tikzpicture}
   \draw (-3,0)--(0,1)--
   (1,0) (-2.7,0)--(0,1);
   \draw[fill=white]  (-3,0) circle (4pt) (0,1)node{\tiny$2$} circle (4pt) (1,0) circle (4pt) (-2.7,0) circle (4pt);
   \draw[-stealth] (1.5,.5)--node[above]{$\on{trop}(\psi)$} (3,.5);
  \draw (3.5,0)--node[below right]{\color{sebgreen}{$1$}}(6.5,1);
  \draw[sebblue] (8.5,0)--node[above right]{\color{sebgreen}{$\frac{3}{2}$}}(6.5,1)--node[right]{$5$}(6.5,1.5) (8.5,0)--node[right]{$1$}(8.5,1.5);
  \draw[fill=white] (3.5,0) circle (4pt)  (6.5,1) circle (4pt)  (8.5,0) circle (4pt);
  \draw[-stealth] (9,.5)--node[above]{\color{sebgreen}{$\lambda_T$}} (10.5,.5);
  
  \draw[-stealth] (11,0)node{$0$}--(11,1.5);
  \end{tikzpicture}
 \end{figure}
 
The line bundle $L$ is trivial on the middle component of the chain, which is therefore contracted to a node. The dualising sheaf $\omega_{\oP}$ is itself a line bundle. Picking a local generator $u$ of $\omega_{\oP}\otimes\OO_{\oP}(-1)$, local equations are determined by the vanishing multiplicity of $2\lambda_T$ at the points adjacent to its support (we may localise to the complement of the branch points away from $\on{Supp}(\lambda_T)$):
\[\hat\OO_{\oC,q}=\kfield[\![s_1,s_2]\!][u]/(s_1s_2,u^2-s_1^2-s_2^3),\]
which is isomorphic to $\kfield[\![x_1,\bar x_1,x_2]\!]/(x_1x_2-\bar x_1x_2,x_1\bar x_1-x_2^3)$ \cite[Equation (4) on p.11]{Bat19} via \[s_1=\frac{1}{2}(x_1-\bar x_1),s_2=x_2,u=\frac{1}{2}(x_1+\bar x_1).\]
\end{example}

\begin{example}
 We construct a $(2)$-tailed ribbon of genus two, c.f. \cite[Definition 2.21]{BatCar}.
 \begin{figure}[h]
  \begin{tikzpicture}
   \draw (.85,0)--(1,1)--(1.15,0);
   \draw[fill=white] (.85,0) circle (4pt)  (1,1)node{\tiny$2$} circle (4pt)  (1.15,0) circle (4pt);
   \draw[-stealth] (1.5,.5)--node[above]{$\on{trop}(\psi)$} (3,.5);
  \draw (3.5,0)--node[right]{\color{sebgreen}{$1$}}(3.5,1);
  \draw[sebblue] (3.5,1)--node[right]{$6$}(3.5,1.5);
  \draw[fill=white] (3.5,0) circle (4pt)  (3.5,1) circle (4pt);
  \draw[-stealth] (4,.5)--node[above]{\color{sebgreen}{$\lambda_T$}} (5.5,.5);
  
  \draw[-stealth] (6,0)node{$0$}--(6,1.5);
  \end{tikzpicture}
 \end{figure}
 
In this case $\oP$ is isomorphic to $P$. Local equations:
\[\hat\OO_{\oC,q}=\kfield[\![s_1,s_2]\!][u]/(s_1s_2,u^2-s_1^2),\]
which is isomorphic to $\kfield[\![x_1,x_2,y]\!]/(x_1x_2,(x_1-x_2)y)$ \cite[Example 2.20]{BatCar} via \[x_1=u+s_1,x_2=u-s_1,y=s_2.\]

\end{example}

\begin{example}
    We construct a non-reduced singularity of genus three.
    
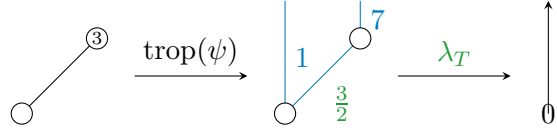
\begin{figure}[h]
  \begin{tikzpicture}
   \draw (0,0)--(1,1);
   \draw[fill=white] (0,0) circle (4pt)  (1,1)node{\tiny$3$} circle (4pt);
   \draw[-stealth] (1.5,.5)--node[above]{$\on{trop}(\psi)$} (3,.5);
  \draw[sebblue] (4.5,1)--node[below right]{\color{sebgreen}{$\frac{3}{2}$}}(3.5,0)--node[right]{$1$}(3.5,1.5) (4.5,1)--node[right]{$7$}(4.5,1.5);
  \draw[fill=white] (3.5,0) circle (4pt)  (4.5,1) circle (4pt);
  \draw[-stealth] (5,.5)--node[above]{\color{sebgreen}{$\lambda_T$}} (6.5,.5);
  
  \draw[-stealth] (7,0)node{$0$}--(7,1.5);
  \end{tikzpicture}
 \end{figure}

 Again $\oP$ is isomorphic to $P=L\cup R$. Local equations:
\[\hat\OO_{\oC,q}=\kfield[\![s_1,s_2]\!][u]/(s_1s_2,u^2-s_1^3).\]
The double structure on the right component $R$ is given by the line bundle $\OO_R(-\frac{1}{2}\mathbf b-\lambda_T)=\OO_R(-2)$, hence the ribbon $\overline{R}$ has genus $1$. We can obtain the curve $\oC$ by gluing the cuspidal curve $\overline{L}$ with the ribbon $\overline{R}$ along a length two subscheme, which checks out to give $p_a(\oC)=3$.
\end{example}

\section{Local computations}\label{sec:local}
Suppose now that $p$ is a point of $\oP$ with $\ell$ branches. Write $s_i$ for a local parameter along the $i$th branch  of $P$ above $p$ (we shall replace these by unit multiples if needed, taking advantage of $\kfield$ being algebraically closed). Local equations of $\oP$ at $p$ are \[A = \hat{\OO}_{\oP,p} = \kfield\llbracket s_1,\ldots, s_l \rrbracket / (s_i s_j : i \neq j).\]

Now write $q$ for the point of $\oC$ above $p$. Recall that the multiplication of the double cover is determined by $\lambda_T$ and $\mathbf b$. Let $m_i$ be the (positive) slope of $2\lambda_T$ at the $i$th branch.  Write
\[
  \delta_i = \begin{cases}
    0 & \text{ if } \lambda_T > 0 \text{ on the $i$th branch} \\
    1 & \text{ else.}
  \end{cases}
\]

Then, if $\ell = 1$ and $p$ is the image of a contracted component we have
\[
  \hat{\OO}_{\oC,q} = A[u] / (u^2 - \delta_1s_1^{m_1 + 2}),
\]
which is either a germ of a ribbon or an $A_{m_1 + 1}$ singularity. The additional 2 in the exponent comes
from the fact that near the unique node $\tilde{p}$ mapping to $p$, the natural map $\tau^*\omega_{\ol{P}} \cong \tau^*\tau_*\omega_{P} \to \omega_{P}$ is induced by twisting at $\tilde{p}$.

If $\ell = 1$ and $p$ belongs to $\mathbf b$,
we have
\[
  \hat{\OO}_{\oC,q} = A[u] / (u^2 - \delta_1s_1),
\]
either a germ of a ribbon or a point of ramification of the cover.

If $\ell = 1$ and $p$ does not belong to $\mathbf b$, we have
\[
  \hat{\OO}_{\oC,q} = A[u] / (u^2 - \delta_1),
\]
either a germ of a ribbon or a trivial part of the double cover.

If $\ell \neq 1$, we have the ring
\[
  \hat{\OO}_{\oC,q} = A[u_2, \ldots, u_\ell] / I
\]
where $u_2, \ldots, u_\ell$ are the images of $\langle \frac{ds_1}{s_1}, -\frac{ds_2}{s_2}, \ldots, 0 \rangle, \ldots, \langle \frac{ds_1}{s_1}, 0, \ldots, -\frac{ds_\ell}{s_\ell} \rangle$ and $I$ is generated by
\begin{enumerate}
  \item $s_1(u_i - u_j)$ for each $2 \leq i < j \leq \ell$;
  \item $s_iu_j$ for each $i \neq j$, $2 \leq i,j \leq \ell$;
  \item $u_i^2 - \delta_1s_1^{m_1} - \delta_is_i^{m_i}$ for each $i = 2,\ldots, \ell$;
  \item $u_iu_j - \delta_1s_1^{m_1}$ for each $i \neq j$ with $2 \leq i,j \leq \ell$.
\end{enumerate}

\subsection{Gluing} The subscheme cut out by $s_i$ for $i \neq 1$ and $u_i - u_j$ for $i < j$ is isomorphic to
\begin{equation} \label{eq:s1_branch}
  \kfield\llbracket s_1 \rrbracket [ u_1 ] / (u_1^2 - \delta_1s_1^{m_1}),
\end{equation}
where $u_1$ is the common image of $u_2,\ldots, u_m$. It is either a germ of a ribbon or an $A_{m_1 - 1}$ singularity.

Similarly, for each $j = 2,\ldots, \ell$, the subscheme cut out by $s_i$ for $i \neq j$ and $u_i$ for $i \neq j$ is isomorphic to
\[
  \kfield \llbracket s_j \rrbracket [u_j] / (u_j^2 - \delta_js_j^{m_j}), 
\]
which is again either a germ of a ribbon or an $A_{m_j - 1}$ singularity.

If we only restrict down to the subscheme cut out by $s_1$, we find that we get
\begin{equation} \label{eq:non_s1_branches}
  \kfield \llbracket s_2,\ldots, s_\ell \rrbracket [u_2, \ldots, u_{\ell}] / (s_is_j, s_i u_j, u_i^2 - \delta_is_i^{m_i} : i \neq j, 2 \leq i, j \leq \ell),
\end{equation}
the transverse union of the singularities for $j = 2, \ldots, \ell$ above.

Our next claim is that the singularity at $q$ is the result of gluing the tangent vector $\frac{\partial}{\partial u_1}$ of Spec of (\ref{eq:s1_branch}) with the tangent vector $\sum_{i = 2}^\ell \frac{\partial}{\partial u_i}$ of Spec of (\ref{eq:non_s1_branches}).

To see this, consider the sequence
\[
  0 \to \hat{\OO}_{\oC,q} \to \frac{\kfield\llbracket s_1 \rrbracket [ u_1 ]}{(u_1^2 - \delta_1s_1^{m_1})} \times \frac{\kfield \llbracket s_2,\ldots, s_\ell \rrbracket [u_2, \ldots, u_{\ell}]}{(s_is_j, s_i u_j, u_i^2 - \delta_is_i^{m_i} : i \neq j, 2 \leq i, j \leq \ell)} \to Q \to 0.
\]
To see that the first map is injective, observe that the kernel is contained in $(s_1) \cap (s_2,\ldots, s_\ell) = 0$. Note that both $\langle s_1, 0 \rangle$ and $\langle 0, s_i \rangle$ for $i = 2, \ldots, \ell$ are in the image of the first map, so $Q$ is supported on $V(s_1,\ldots, s_\ell)$. Restricting to this vanishing, we find
\[
  0 \to k[u_2,\ldots, u_\ell]/(u_2,\ldots, u_\ell)^2 \to k[u_1]/u_1^2 \times k[u_2,\ldots, u_\ell] \to Q \to 0.
\]
The first map clearly admits a retract, so we conclude $Q \cong k[\epsilon]/\epsilon^2$. This yields the claim.

\subsection{Normalisation}
 The normalisation $\oC^\nu$ of the germ of $\oC$ at $q$ can be computed as follows. Consider:
 \bcd
 \oC^\nu\ar[r] & \oC_{\oP^\nu,\text{red}} \ar[r,hook] & \oC_{\oP^\nu}\ar[r]\ar[d]\ar[dr,phantom,"\Box"] & \oC\ar[d]\\
  & & \oP^\nu\ar[r] & \oP
 \ecd
 From this we see that it is enough to understand the normalisation of the $A_m$-singularities and of ribbons, and put these formulae together. Assume that $\delta_1=1$ and $m_1$ is even; the other cases are left to the avid reader. Renumbering $\{2,\ldots,\ell\}$ we may assume that:
 \begin{itemize}
  \item for $i=2,\ldots,h$ we have $\delta_i=1$ and $m_i$ even,
  \item for $i=h+1,\ldots,k$ we have $\delta_i=1$ and $m_i$ odd,
  \item for $i=k+1,\ldots,\ell$ we have $\delta_i=0$.
 \end{itemize}
The normalisation is then given by the ring:
\[\prod_{i=1}^h\kfield\llbracket a_i\rrbracket\times\kfield\llbracket b_i\rrbracket\times\prod_{i=h+1}^k\kfield\llbracket c_i\rrbracket\times\prod_{i=k+1}^l\kfield\llbracket d_i\rrbracket\]
with ring homomorphism:
\begin{equation} 
 s_i\mapsto\left.
    \begin{cases}
            (a_i,b_i)&\text{for }i=1,\ldots,h\\
            c_i^2&\text{for }i=h+1,\ldots,k\\
            d_i&\text{for }i=k+1,\ldots,\ell
    \end{cases}
 \right. \qquad u_i\mapsto \left.
  \begin{cases}
           (a_1^{m_1/2},-b_1^{m_1/2},a_i^{m_i/2},-b_i^{m_i/2})&\text{for }i=2,\ldots,h\\
           (a_1^{m_1/2},-b_1^{m_1/2},c_i^{m_i})&\text{for }i=h+1,\ldots,k\\
           (a_1^{m_1/2},-b_1^{m_1/2},0)&\text{for }i=k+1,\ldots,\ell\\
  \end{cases}\right.
\end{equation}

\subsection{Differentials}

For this section we assume that $\ol{C}$ is reduced.

\begin{lemma}
 A local generator of $\omega_{\oC}$ at $q$ is
 \[\eta=\frac{\on{d}\!s_1}{us_1}-\sum_{i=2}^{\ell}\frac{\on{d}\!s_i}{u_is_i}.\]
\end{lemma}
\begin{proof}
 Follows from the above equations and Rosenlicht's description of differentials.
\end{proof}

From this we can read off the pullback of the canonical bundle of $\oC$ to its normalisation $\nu\colon C^{\nu}\to \oC$ (where $C^{\nu}$ can be viewed as a subcurve of $C$). 
\begin{corollary}\label{cor:conductor}
 Write $(C^{\nu},q_i,\bar q_i,q_j)_{\substack{i=1,\ldots,h\\j=h+1,\ldots,k}}$ for the pointed normalisation of $\oC$ at $q$. Then:
    \[\nu^*\omega_{\oC}=\omega_{C^{\nu}}\left(\sum_{i=1}^h(\frac{m_i}{2}+1)(q_i+\bar q_i)+\sum_{i=h+1}^k (m_i+1)q_i\right).\]
\end{corollary}
From this and Noether's formula we see that the conductor ideal of $\nu$ is:
\begin{equation}\label{eqn:conductor}
    \mathfrak{c}=\left(a_i^{\frac{m_i}{2}+1},b_i^{\frac{m_i}{2}+1},c_i^{m_i+1}\right).
\end{equation}

From this we can verify that $\oC$ is Gorenstein in a second way: it is well-known that $\oC$ is Gorenstein if and only if $\dim_{\kfield}(\OO_{C^\nu}/\mathfrak c)=2\delta$ (see e.g. \cite[Proposition VIII.1.16]{AltmanKleiman}).

Now,
\begin{equation}\label{eqn:genus_of_singularity}
    \delta=g(q)+2h+k-1,
\end{equation}
where $2h+k$ is the number of branches of $q$, and $g(q)$ is the genus of the singularity.

The latter is the same as the genus of the subcurve of $C$ contracted to it. This corresponds to a connected component of the support of $\lambda_T$. Since the following formulae are stable under edge contraction, we may as well assume that there is a single vertex $v$ in the support of $\lambda_T$, resp. irreducible component $C_v$ of $C$ contracting to $q$. The genus of this component is determined by the Riemann--Hurwitz formula:
\begin{equation}\label{eqn:RH}
    2g(C_v)+2=b+k,
\end{equation}
where $b$ is the number of branch points supported on $C_v$, and $k$ the number of odd nodes (for $\mathbf b$) adjacent to $v$.
On the other hand, balancing $\lambda_T$ at $v$ as in equation \eqref{eqn:balancing}, we find:
\begin{equation}\label{eqn:balancing_applied}
\on{div}(\lambda_T)=\sum_{i=1}^{h+k}\frac{m_i}{2}=\val(v)-2+\frac{b}{2}.\end{equation}
Finally, from the above formula for the conductor we find that 
 \begin{align*}
     \dim_{\kfield}(\OO_{C^\nu}/\mathfrak c)&=\sum_{i=1}^{h+k}m_i+2h+k &\text{by eq. \eqref{eqn:conductor}}\\
      &=2\val(v)-4+b+2h+k &\text{by eq. \eqref{eqn:balancing_applied}}\\
      &=b+k-4+2(2h+k) &\text{by a simple manipulation}\\
      &=2(g+2h+k-1)=2\delta. &\text{by eq. \eqref{eqn:RH} and eq. \eqref{eqn:genus_of_singularity}}
 \end{align*}

\section{Classification of Gorenstein hyperelliptic curves}
In this section we prove a partial converse to our previous result, namely that most Gorenstein 
hyperelliptic curves arise from our construction. We focus on the unmarked case for notational simplicity. We start by specifying what exactly we mean by a Gorenstein hyperelliptic curve.

\begin{definition}
 We say that $\opsi\colon\oC\to\oP$ is a \emph{Gorenstein hyperelliptic cover} if $\oP$ is a rational, reduced, Cohen--Macaulay projective curve; $\opsi$ is a finite (not necessarily flat) cover of degree two over every irreducible component of $\oP$; $\oC$ is a Gorenstein (not necessarily reduced) curve; there is a hyperelliptic involution $\bar\iota$ on $\oC$ with quotient $\oP$.
\end{definition}

\begin{remark}
    Every non-reduced component of $\oC$ is a necessarily \emph{split} ribbon \cite[\S1]{BayerEisenbud}.
\end{remark}

\begin{theorem}\label{thm:converse}
 Every smoothable Gorenstein hyperelliptic cover arises from the construction of \S \ref{sec:construction}.
\end{theorem}

\begin{corollary}\label{cor:converse_reduced}
 Every reduced Gorenstein hyperelliptic cover arises from the construction of \S \ref{sec:construction}.
\end{corollary}

In presence of a $G=\ZZ/2\ZZ$-action on $\oC$, we may split the structure (in fact, any equivariant) sheaf into eigenspaces for the $G$-action (on every $G$-stable open). We can thus write
\[\opsi_*\OO_{\oC}=\OO\oplus\Mcal,\]
where $\OO$ denotes the $1$-eigenspace, and $\Mcal$ the $-1$. By assumption, $\bar\iota$-invariant functions descend to $\oP$, whence we can identify $\OO$ with $\OO_{\oP}$. In particular, the finite cover $\opsi$ admits a trace map even when it is not flat. 

We know that $\Mcal$ is some sheaf of pure rank one on $\oP$. Our next goal is to show that $\Mcal$ is a twist of $\omega_{\oP}$ by a line bundle. We recall that, in his study of generalised divisors, Hartshorne has introduced a generalisation of reflexivity for sheaves which is useful when the base scheme is not Gorenstein. Denote by $-\omd$ the functor $\mathcal{H}om(-,\omega)$, i.e. $\omega$-dualisation. A sheaf $\Fcal$ is \emph{$\omega$-reflexive} if $\Fcal\to\Fcal^{\omega\omega}$ is an isomorphism. This implies that $\Fcal$ is torsion-free \cite[Lemma 1.4]{HarGenDiv}.

\begin{remark}
 It follows from Grothendieck's duality for a finite morphism $f\colon X\to Y$ that
 \[(f_*\Fcal)\omd=\homs_{\OO_Y}(f_*\Fcal,\omega_Y)=f_*\homs_{\OO_X}(\Fcal,f^!\omega_Y)=f_*(\Fcal\omd).
 \]
\end{remark}

In particular, 
$\opsi_*\omega_{\oC}=\homs_{\OO_{\oP}}(\opsi_*\OO_{\oC},\omega_P)=\omega_{\oP}\oplus\Mcal\omd$.

\begin{remark}\label{rmk:antiinvariant2}
Since $\oP$ is rational, $\omega_{\oP}$ has no global section (by Serre duality). It follows that (global regular) sections of the dualising sheaf on $\oC$ can be identified with sections of $\Mcal\omd$ on $\oP$, in particular they are all $\iota$-anti-invariant. This generalises Remark \ref{rmk:antiinvariant1} beyond the case of smooth curves.
\end{remark}

Since $\OO_{\oC}$ and $\OO_{\oP}$ are both $\omega$-reflexive, we may conclude that the same holds true for $\Mcal$.

\begin{lemma}
 $\Mcal$ is a rank-one, $\omega$-reflexive sheaf.
\end{lemma}
\begin{lemma}
 $\Mcal\omd$ is a line bundle, except where $\opsi$ maps a node to a node with ramification.
\end{lemma}
\begin{proof}
 We may work locally around a closed point $p$ of $\oP$.
 If $\oP$ is smooth at $p$, then $\opsi$ is flat by ``miracle flatness'', so $\Mcal$ is itself a line bundle, and $\Mcal\omd$ is as well.

 If $p$ is a node, we consider two cases: either $\opsi$ is flat over $p$, in which case we can conclude as before\footnote{All singularities of the form $\kfield\lbb x,y,z\rbb/(xy,z^2-x^\alpha-y^\beta)$ fall under this category, e.g. $D_k$-singularities when $\alpha=1$.}; or $\opsi$ is not flat. In this case we claim that $\oC$ has a node at the preimage $n$ of $p$, and $\opsi$ is ramified at $n$ on both branches. To show this, we are going to normalise $\oP$ and $\oC$ simultaneously. Indeed, $\Mcal$ is not a line bundle, but there is a line bundle $\Mcal'$ on the normalisation $\nu\colon \oP'\to\oP$ such that $\Mcal=\nu_*\Mcal'$\cite[Proposition 10.1]{OdaSheshadri}. Consider the following exact sequence:
\begin{equation}\label{eqn:normalisation}
    0\to \OO_{\oP}\oplus\Mcal\to \nu_*(\OO_{\oP'}\oplus\Mcal')\to \kfield_p\to 0
\end{equation}
 
We may endow the second term with a $\nu_*\OO_{\oP'}$-algebra structure induced by the one of $\OO_{\oC}$. Indeed, 
there is always a map:
\[(\nu_*\Mcal')^{\otimes 2}\to \nu_*(\Mcal'^{\otimes 2}),\]
 which in this case a local computation shows to be surjective. In fact, we may as well replace $(\nu_*\Mcal')^{\otimes 2}$ by $\on{Sym}^2(\nu_*\Mcal')$. We get the desired multiplication map by lifting:
 \bcd
\on{Sym}^2\Mcal\ar[r,"\mu"]\ar[d,twoheadrightarrow] & \OO_{\oP}\ar[d] \\
\nu_*(\Mcal'^{\otimes 2})\ar[r,dashed,"{\mu'}"] & \nu_*\OO_{\oP'}
 \ecd
 Explicitly, if $\Mcal'$ is generated as $\OO_{\oP'}$-module by an element $(x,y)$ (and its pushforward along $\nu$ is generated as an $\OO_{\oP}$-module by two elements $x=(1,0)\cdot(x,y)$ and $y=(0,1)\cdot(x,y)$), its square $\Mcal'^{\otimes 2}$ is generated by $(x^2,y^2)$ as an $\OO_{\oP'}$-module, and by $x^2$ and $y^2$ as an $\OO_{\oP}$-module. The $\OO_{\oP}$-module $\on{Sym}^2\Mcal$ has an extra generator $xy$. On the other hand, locally, $\OO_{\oP}\simeq \kfield[s,t]/(st)$ (while $\OO_{\oP'}\simeq \kfield[s]\oplus\kfield[t]$), and $sy=tx=0$ implies that $\mathfrak{m}_p \cdot xy=0$, hence the multiplication map $\mu\colon\on{Sym}^2\Mcal\to\OO_{\oP}$ must send this element to $0$. It follows that the multiplication map $\mu'\colon \nu_*(\Mcal'^{\otimes 2})\to \nu_*\OO_{\oP'}$ is well-defined. Moreover, it clearly lifts to a map of $\OO_{\oP'}$-modules.
We thus get the desired double cover $\oC'\to\oP'$, together with a birational morphism $\oC'\to\oC$. Since $\oP'$ is smooth (disconnected), the former map is flat and $\oC'$ is smooth, so the latter map is the normalisation of $\oC$. Equation \eqref{eqn:normalisation} shows that $\oC$ has $\delta$-invariant $1$ (and at least two branches) at $n$, so $n$ must be a node, and moreover the cover is ramified at $n$ on both branches.

Finally, if $\oP$ is not Gorenstein at $p$ we may argue as follows. Let $q$ be the point of $\oC$ over $p$ (if there were two, $\opsi$ would be a local isomorphism, contradicting the fact that $\oC$ is Gorenstein). The group $G$ acts on $\omega_{\oC}$. By assumption, $\omega_{\oC}$ admits a single generator at $q$ that we will call $\eta$. Consider the eigenspace decomposition $\eta=\eta_1+\eta_{-1}$. If $\eta_{-1}=0$, then $\omega_{\oP}$ is generated by $\eta_1$ as an $\OO_{\oP}$-module, which is a contradiction. Since $\eta$ generates 
$\omega_{\oC}$ and $\eta_{-1}$ is itself a section of $\omega_{\oC}$, we can write $\eta_{-1}=f\eta$. We claim that $f(q)\neq 0$, so we can as well take $\eta_{-1}$ as a generator of $\omega_{\oC}$. Decomposing $f$ and $\eta$ into their homogeneous pieces, we write:
\[\eta_{-1}=f_{-1}\eta_1+f_1\eta_{-1}.\]

Since $\bar\iota^*f_{-1}(q) = -f_{-1}(q)$, which implies $f_{-1}\in \mathfrak{m}_q$, we have to check that $f_1(q)\neq 0$. Were $f_1(q) = 0$, then $1 - f_1$ would be a unit, and we could write:
\[\eta_{-1}=\frac{f_{-1}}{1-f_1(q)}\eta_1,\]
so we could take $\eta_1$ as a generator of $\omega_{\oC}$, which is a contradiction as above. This shows that the generator of $\omega_{\oC}$ can be assumed to be of pure weight $-1$, hence $\Mcal^\omega$ has a single generator as an $\OO_{\oP}$-module.
\end{proof}

\begin{remark}
    As in the previous section, the failure of $\Mcal\omd$ to be a line bundle can be cured by introducing an orbifold structures at the faulty nodes, a passage that we leave to the avid reader.
\end{remark}

\begin{lemma}
 $\opsi^*\Mcal\omd\simeq\omega_{\oC}.$
\end{lemma}
\begin{proof}
 By adjunction there exists a morphism \[\opsi^*\Mcal\omd\to\opsi^*\opsi_*\omega_{\oC}\to\omega_{\oC}.\]
 Since $\opsi$ is finite, it is enough to check that the composite is an isomorphism after pushing forward along $\opsi$. Since $\Mcal\omd$ is a line bundle, we may apply the projection formula to compute:
 \[\opsi_*\opsi^*\Mcal\omd=\Mcal\omd\otimes\opsi_*\OO_{\oC}=\Mcal\omd\otimes(\OO_{\oP}\oplus\Mcal).\]
 Since $\Mcal\omd$ is a line bundle, $\Mcal\omd\otimes\Mcal$ is also a rank-one torsion-free, hence we have a short exact sequence:
 \[0\to\Mcal\omd\otimes\Mcal\xrightarrow{\on{ev}}\omega_{\oP}\to\Qcal\to0,\]
 where $\Qcal$ is a torsion sheaf. By taking $\omega$-duals, we get:
 \[0\to\homs(\omega_{\oP},\omega_{\oP})=\OO_{\oP}\to\homs(\Mcal\omd\otimes\Mcal,\omega_{\oP})\to\mathcal{E}xt^1(\Qcal,\omega_{\oP})\to0.\]
 Since $\Mcal\omd$ is a line bundle, the first arrow is an isomorphism, which shows that $\Qcal$ vanishes. We conclude that
 \[\opsi_*\opsi^*\Mcal\omd=\Mcal\omd\oplus\Mcal\omd\otimes\Mcal=\Mcal\omd\oplus\omega_{\oP}=\opsi_*\omega_{\oC}.\]
\end{proof}

\begin{proof}[Proof of Theorem \ref{thm:converse}]
Consider a smoothing $\bar\Psi\colon\oCcal\to\oPcal$ of $\opsi$ over $\dvr$, and mark the generic fibre $\oPcal_\eta$ with the branch divisor of $\bar\Psi$. 
After a finite base change if necessary, let $(\Pcal,\Bcal)$ be the unique limit of $(\oPcal_{\eta},\overline\Bcal_{\eta})$ as a stable curve with unordered markings. Let $\Psi\colon\Ccal\to\Pcal$ be the associated hyperelliptic admissible cover with the minimal log structure. Let $\Phi_P\colon\Pcal\to\oPcal$ denote the contraction, and similarly $\Phi_C$.

Since $\Ccal$ is a normal surface, and by reflexivity of the sheaves involved, we notice that $\omega_{\Ccal/\dvr}$ and $\Phi_C^*\omega_{\oCcal/\dvr}$ differ only by a vertical divisor, supported on the central fibre. We may hence write:
\[\Phi_C^*\omega_{\oCcal/\dvr}=\omega_{\Ccal/\dvr}(\lambda),\]
for some conewise-linear function $\lambda\in H^0(\Ccal,\overline{M}_\Ccal)$, a priori only with the divisorial log structure of $\Ccal$ with respect to its central fibre. On the other hand, let $\Pcal^{\rm tw}$ denote the orbicurve $[\Ccal/\iota]$. Since $\omega_{\oC}=\opsi^*\Mcal\omd$, and $\omega_C=\psi^*\omega_{P^{\rm tw}}$, we deduce that their difference is also pulled back from $P^{\rm tw}$. Hence $\lambda$ is pulled back from $\lambda_T$ on $\Pcal^{\rm tw}$ with its divisorial log structure.

We may now apply our construction to $(\Psi,\lambda_T)$, thus obtaining a Gorenstein hyperelliptic curve $\oPsi'\colon\oCcal'\to\oPcal'$, fitting in the following diagram:
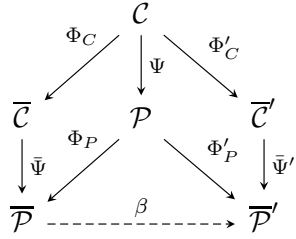
\bcd
&\Ccal\ar[dl,"\Phi_C" above=.2cm]\ar[d,"\Psi"]\ar[dr,"\Phi_C'"]&\\
\oCcal\ar[d,"\oPsi"]&\Pcal\ar[dl,"\Phi_P" above=.2cm]\ar[dr,"\Phi_P'"]&\oCcal'\ar[d,"\oPsi'"]\\
\oPcal\ar[rr,dashed,"\beta"] &&\oPcal'\\
\ecd

Observe that $\oPcal$ and $\oPcal'$ are normal surfaces, and the exceptional loci of $\Phi_P$ and $\Phi_P'$ are the same, 
so we may find an isomorphism $\beta\colon\oPcal\simeq\oPcal'$ commuting with the $\Phi_P$'s by birational rigidity \cite[Lemma 1.15]{Debarre}. Moreover, $\beta^*\OO_{\oPcal'}(\mathbbm 1)\simeq\Mcal\omd$.

Now, by $\omega$-reflexivity, we recover $\Mcal\simeq\beta^*\omega_{\oPcal}(-\mathbbm 1)$, and therefore $\oPsi_*\OO_{\oCcal}=\beta^*(\oPsi'_*\OO_{\oCcal'})$. 
The branch divisor is determined by the image of $\mathcal B$, and by the components of the central fibre that are contained in the support of $\lambda$ without being contracted by $\Phi$ (by Riemann--Hurwitz), hence we conclude that there is also an isomorphism $\alpha\colon\oCcal\simeq\oCcal'$ covering $\beta$.
\end{proof}

\begin{proof}[Proof of Corollary \ref{cor:converse_reduced}]
We are left to show that $\opsi$ can be smoothed out when $\oC$ is reduced. We will proceed step by step by showing that the various ingredients of this moduli problem are unobstructed.

The reduced rational curve $\oP$ is smoothable, see \cite[Example 29.10.2]{HarDef}.

The line bundle $\Mcal\omd$ can be smoothed out since the relative Picard scheme of a curve is unobstructed. Consequently, the structure sheaf of $\oC$ can be smoothed out by taking its $\omega$-dual $\Mcal$.

Finally, the multiplication map $\mu$ is a cosection of $\Mcal^{\otimes2}$. 
Consider the pairing 
\[H^0(\Mcal^{\otimes -2})\times Hom(\Mcal^{\otimes -2},\omega_{\oP})\to H^0(\omega_{\oP})=0\]
by composition. Since $\oC$ is reduced by assumption, $\mu$ does not vanish generically on any component of $\oP$. It follows that every section of $\mathcal{H}om(\Mcal^{\otimes -2},\omega_{\oP})$ must vanish generically, and since this sheaf is torsion-free, it is zero \emph{tout court}. By Serre duality $h^1(\Mcal^{\otimes -2})=0$, hence deformations of $\mu$ are also unobstructed.

\end{proof}

\begin{remark}
 We expect the result to hold for all Gorenstein hyperelliptic curves, but we have not been able to prove the smoothability of non-reduced curves yet. On the other hand, these curves can be dispensed with as far as our application to differentials is concerned.
\end{remark}

\section{The differential descent conjecture}\label{sec:GRC}

\subsection{Abelian differentials in general} The moduli space of Abelian differentials has at most three connected components (depending on the multiplicity $\mu$ of the zeroes) \cite{KontsevichZorich}. In general, connected components of the space of multiscale differentials are not irreducible. The \emph{global residue condition} (GRC) was introduced in \cite{BCGGM16} to single out the \emph{smoothable} differentials. Roughly speaking, it says that the sum of the residues at poles of level $i$ that are joined by a connected subcurve at level $i+1$ must vanish, despite the possibility that the corresponding nodes belong to different subcurves at level $i$. The proof of necessity goes by cutting the generic fibre of a smoothing along the vanishing cycle corresponding to these nodes, and applying Stokes' theorem to compute the integral of the abelian differential on the resulting surface with boundary. The proof of sufficiency is more complicated, and based on a refined \emph{plumbing} construction. With the logarithmic understanding of the moduli space of \emph{generalised} multiscale differentials reached in \cite{Chen2,Tale}, the GRC remains the only ingredient of \cite{BCGGM} relying on transcendental techniques. A purely algebraic description of smoothable differentials is contained in the following conjecture, originally due to Ranganathan and Wise. 
\begin{conjecture}[Gorenstein curves and smoothable differentials]\label{conj} \leavevmode
Let $(C,\eta)$ be a logarithmic rubber differential with tropicalisation $\blambda$. Then $\eta$ is smoothable if and only if \begin{enumerate}[label=(\roman*)]
    \item for every level $i$, the truncation $\lambda_i$ of $\blambda$ (as in \S \ref{sec:levels}) is a realisable tropical differential;
    \item there exists a logarithmic modification $\widetilde{C}\to C$, a natural extension $\tilde\eta$ of the pullback of $\eta$ to $\widetilde C$, and a reduced Gorenstein contraction $\sigma\colon \widetilde C\to\oC_i$ such that $\sigma^*\omega_{\oC_i}=\omega_C(\lambda_i)$, and
    \item the differential $\tilde\eta_i$ at level $i$ descends to a local generator of $\omega_{\oC_i}$. 
\end{enumerate}
\end{conjecture}

Here $\widetilde{C}_i$ is determined by $\eta$ as follows, in order to ensure that the twist of the canonical bundle be trivial on the upper levels, and to avoid non-reduced components in the contractions $\oC_i$. Indeed, ribbons appear when $\omega_C(\lambda_i)$ has positive degree on the support of $\lambda_i$. This happens precisely when at least one zero of order $m\geq1$ is contained in the support of $\lambda_i$. In this case, since we have a non-trivial logarithmic structure of marking type at the zero, we can subdivide the corresponding leg at level $i$; classically, this means sprouting a new semistable rational component at the marking. In the natural coordinates $[x_0:x_1]$ with respect to the two special points, the differential $\eta$ can be extended uniquely to the new component $\tilde{v}$ by setting $\eta_{\tilde{v}}=x_0^{m}\on{d}\!x_0$; the choice of a non-zero scalar is compensated by the automorphisms of the underlying curve. Notice that this differential does not contribute to the GRC, since $m+2>1$. The mere existence of the non-zero differentials at levels higher than $i$ guarantees that the twist of the canonical bundle by $\lambda_i$ will be trivial (not just numerically). We provide the following ad hoc example in the hope of acquainting the reader with the log modification procedure.

\begin{example}
    Let $(C,\eta)$ be a generalised multiscale differential, where $C$ consists of two components $C_0$ and $C_{-1}$ joined at a single node $q$. Assume that $C_0$ is a curve of genus two, and $\eta_0$ is a holomorphic differential with simple zeroes at $q$ and its conjugate point $\bar q$, which in particular is a marking of $C$ (note that $C$ is not hyperelliptic in the sense of admissible covers, although $C_0$ is; the specific $C_1$ will be immaterial for this discussion).  In a general one-parameter smoothing, $C_0$ will have negative self-intersection; in particular, it can be contracted by general principles (Artin's criterion). The resulting singularity is formally isomorphic to $\kfield[\![t^3,t^4,t^5]\!]$, which is not Gorenstein. Indeed, since $\omega_{C_0}=\OO_{C_0}(q+\bar q)$, twisting by a multiple of $C_0$ will never make the relative dualising bundle of the family trivial on $C_0$. Instead, we are going to modify $C$ by log blowing up $C_0$ at $\bar q$, and \emph{then} contract, which results into a locally planar singularity of type $A_5$, whose dualising bundle is generated by a meromorphic differential with poles of order three on either branch.
\end{example}

\subsection{Hyperelliptic differentials} The connected component consisting of hyperelliptic differentials is already irreducible \cite[Proposition 5.16]{Chen2}. This is proved by identifying the moduli space of hyperelliptic differentials with a moduli space of quadratic differentials on rational curves. We therefore view the following result as a first proof of concept for Conjecture \ref{conj}.

\begin{proposition}\label{prop:hypdiff}
 Let $(\psi\colon C\to P, \eta)$ be a log rubber hyperelliptic differential with tropicalisation $\blambda$. The differential $\eta_i$ at level $i$ descends to a generator of the dualising sheaf of the Gorenstein contraction associated to $(\psi\colon \widetilde{C}_i\to P,\lambda_i)$ as in \S\ref{sec:construction}.
\end{proposition}

 
 \begin{proof}
 Let $\eta_i$ denote the collection of differentials on components at level $\leq i$. Then $\eta_i$ is a section of the restriction of $\omega_C(\lambda_i)$ to $C_{\leq i}$, i.e. a meromorphic differential with poles along the level $[i,i+1]$-nodes, whose order of pole is determined by the slopes of $\blambda$ plus one. Since $\eta_i$ is $\iota$-anti-invariant, it descends to a section of the odd part of $\psi_*\omega_C(\lambda_i)$ on $P_{\leq i}$, which is the restriction of $L$. Since the latter is trivial on $P_{>i}$, this section extends uniquely to $P$. We can therefore identify it with a section of $\OO_{\oP}(\mathbbm 1)$ on $\oP$, and in turn with an anti-invariant section of $\omega_{\oC}$. It is in fact a local generator because it has maximal vanishing order, c.f. Corollary \ref{cor:conductor}.
\end{proof}

\begin{remark}
    Anti-invariance under $\iota$ implies that residues at conjugate (resp. Weierstrass) points are opposite (resp. zero). In particular, $\iota$-anti-invariance implies that the Global Residue Condition holds.
    Although every holomorphic differential on a hyperelliptic curve is $\iota$-anti-invariant (Remark \ref{rmk:antiinvariant2}) (and so are their limits), generalised multiscale differentials are only meromorphic on lower levels of the curve, hence
    the above is a proof of Conjecture \ref{conj} not for any differential on a hyperelliptic curve $C$ (in the sense of admissible covers), but only for the $\iota$-anti-invariant ones. See Example \ref{exa:nonhyp}.
\end{remark}
\begin{example}
    Let $C$ be a nodal curve consisting of two hyperelliptic components $C_0$, of genus $g_0$, and $C_{-1}$, joined at a single node $q$, which is Weierstrass on both. Let $\eta$ be an anti-invariant multiscale differential on $C$, such that $\eta_0$ has a single zero of multiplicity $2g_0-2$ at $q$. Then $\lambda$ has slope $2g_0-1$ along the corresponding edge. The meromorphic differential $\eta_1$ has a pole of order $2g_0$ at $q$; notice that the GRC is automatically satisfied by the Residue Theorem. The contraction $\oC_{-1}$ has an $A_{2g_0}$-singularity (of genus $g_0$) at $q$, and $\eta_{-1}\approx \frac{\on{d}\!t}{t^{2g_0}}$ descends to a generator of $\omega_{\oC_{-1}}$.
\end{example}

\begin{example}\label{exa:nonhyp}
    Let $(C,\eta)$ be a genus $3$ hyperelliptic multiscale differential whose level graph is the following:

 
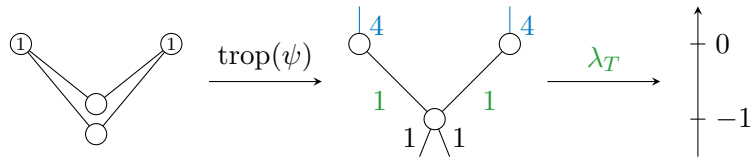
\begin{figure}[h]
  \begin{tikzpicture}
   \draw (-1,1)--(0,-.2)--(1,1)--(0,.2)--(-1,1);
   \draw[fill=white]  (-1,1)node{\tiny$1$} circle (4pt) (1,1)node{\tiny$1$} circle (4pt) (0,-.2) circle (4pt) (0,.2) circle (4pt);
   \draw[-stealth] (1.5,.5)--node[above]{$\on{trop}(\psi)$} (3,.5);
  \draw (3.5,1)--node[below left]{\color{sebgreen}{$1$}}(4.5,0)--node[below right]{\color{sebgreen}{$1$}}(5.5,1)
  (4.3,-.5)--node[left]{{$1$}}(4.5,0)--node[right]{{$1$}}(4.7,-.5);
  \draw[sebblue] (3.5,1)--node[right]{$4$}(3.5,1.5) (5.5,1)--node[right]{$4$}(5.5,1.5);
  \draw[fill=white] (3.5,1) circle (4pt)  (5.5,1) circle (4pt)  (4.5,0) circle (4pt);
  \draw[-stealth] (6,.5)--node[above]{\color{sebgreen}{$\lambda_T$}} (7.5,.5);

  \draw (7.9,0)--(8.1,0) node[right]{$-1$}
  (7.9,1)--(8.1,1) node[right]{$0$};
  \draw[-stealth] (8,-.5)--(8,1.5);
  \end{tikzpicture}
 \end{figure}
Let $\eta$ restrict to $\on{d}\!z$ on the two elliptic curves. Choose coordinates on the rational curves $R_i$ at level $-1$ in such a way that the nodes are $0$ and $\infty$, and the zeroes $a$ and $b$. Then $\eta$ restricts to:
\[\alpha_i(t-a)(t-b)\frac{\on{d}\!t}{t^2}\]
on the rational curve $R_i,\ i=1,2.$ Here, $\iota$-anti-invariance forces $\alpha_1=-\alpha_2$. Contracting the subcurves at level $0$ we obtain two rational curves joined at two tacnodes. There is a linear condition for a meromorphic differential with poles of order two on the pointed normalisation to descend to the tacnode (c.f. \cite[\S 2.2]{Smyth}) which is analogous to the condition $\alpha_1=-\alpha_2$ from above. On the other hand, any choice of $\alpha_i$ gives rise to a generalised multiscale differential.
\begin{remark}
    If $a=-b$, the residues are zero. By varying the $\alpha_i$, we thus get an example of a multiscale differential which satisfies the GRC but is not anti-invariant. This will be the limit of differentials on smooth, non-hyperelliptic curves. There are of course even more examples of non-hyperelliptic differentials on a hyperelliptic curve if we do not impose that the sets of zeroes and poles are invariant under the hyperelliptic involution.
\end{remark}
\end{example}

\bibliographystyle{alpha}
\bibliography{Bibliography.bib}


\end{document}